\documentclass[aop]{imsart}

\makeatletter
\def\@serieslogo{}
\def\@issueinfo{}
\def\@copyrightyear{}
\def\@copyrightline{}
\def\@PII{}
\makeatother

\RequirePackage{amsthm,amsmath,amsfonts,amssymb}
\RequirePackage[authoryear,round]{natbib}%% author-year citations for AOP
\RequirePackage[colorlinks,citecolor=blue,urlcolor=blue]{hyperref}
\RequirePackage{graphicx}

\startlocaldefs
\theoremstyle{plain}

\newtheorem{theorem}{Theorem}[section]
\newtheorem{lemma}[theorem]{Lemma}
\theoremstyle{definition}
\newtheorem{definition}[theorem]{Definition}

\endlocaldefs

\usepackage{mathrsfs}
\usepackage{enumerate}

\newtheorem{proposition}[theorem]{Proposition}
\newtheorem{corollary}[theorem]{Corollary}
\theoremstyle{definition}

\newtheorem{remark}[theorem]{Remark}
\newtheorem{maintheorem}[theorem]{Main Theorem}
\numberwithin{equation}{section}

\newcommand{\calL}{\mathcal{L}}

\usepackage{tikz}

\begin{document}

\begin{frontmatter}
\title{Statistical Limit Theorems for Axiom A Diffeomorphisms: \\
Exponential Mixing, Central Limit Theorem, \\and Large Deviations}
\runtitle{Statistical Limit Theorems for Axiom A Diffeomorphisms}
\runauthor{A. Thiam}

%\begin{aug}
%%%%%%%%%%%%%%%%%%%%%%%%%%%%%%%%%%%%%%%%%%%%%%%%
%%% Only one address is permitted per author. %%
%%% Only division, organization and e-mail is %%
%%% included in the address.                  %%
%%% Additional information such as            %%
%%% identifying the corresponding author must %%
%%% be included in in the Acknowledgments     %%
%%% section if necessary.                     %%
%%% ORCID can be inserted by command:         %%
%%% \orcid{0000-0000-0000-0000}               %%
%%%%%%%%%%%%%%%%%%%%%%%%%%%%%%%%%%%%%%%%%%%%%%%%
%\author[A]{\fnms{Abdoulaye}~\snm{Thiam} \ead[label=e1]{athiam@allenuniversity.edu}}
%
%%%%%%%%%%%%%%%%%%%%%%%%%%%%%%%%%%%%%%%%%%%%%%%
%%% Addresses                                %%
%%%%%%%%%%%%%%%%%%%%%%%%%%%%%%%%%%%%%%%%%%%%%%%
%\address[A]{Division of Mathematics and Natural Sciences, Allen University, Columbia, South Carolina 29204, USA \textbf{}\printead[presep={ ,\ }]{e1}}
%
%\end{aug}

\begin{aug}

\author{\fnms{Abdoulaye}~\snm{Thiam}}

\address{Division of Mathematics and Natural Sciences, Allen University, Columbia, South Carolina 29204, USA\\
\textrm{E-mail:}~\href{mailto:athiam@allenuniversity.edu}{\text{athiam@allenuniversity.edu}}}

\end{aug}

%\begin{abstract}
%We establish statistical limit theorems for equilibrium states of Axiom A diffeomorphisms, derived from the spectral gap of the Ruelle transfer operator (Part~I) transferred to smooth dynamics through the Markov partition coding (Part~III). This Part contains five Main Theorems: the Volume Lemma with explicit two-sided bounds on the Riemannian volume of dynamical Bowen balls; exponential decay of correlations with explicit rates derived from the spectral gap; the Central Limit Theorem with Berry-Esseen bounds at the optimal rate, with an explicit spectral formula for the asymptotic variance and a characterization of its degeneracy through the Liv\v{s}ic coboundary condition; the Almost Sure Invariance Principle providing pathwise Brownian approximation with polynomial error; and a large deviations principle with rate function given by the Legendre transform of the pressure. The individual results are due to Sinai, Ruelle, Ratner, Denker-Philipp, Gou\"{e}zel, Kifer, Melbourne-Nicol, and Young; the contribution is their derivation from a single spectral mechanism with explicit dependence on hyperbolicity data. This Part constitutes Part~V of a six-part series.
%\end{abstract}

\begin{abstract}
\hspace{-0.6cm}We establish statistical limit theorems for equilibrium states of Axiom~A diffeomorphisms, derived from the spectral gap of the Ruelle transfer operator established in Part~I \cite{Thiam2026a} and transferred to smooth dynamics through the Markov partition coding of Part~III \cite{Thiam2026c}. This Part contains five Main Theorems. The first proves the Volume Lemma with explicit two-sided bounds on the Riemannian volume of dynamical Bowen balls in terms of Birkhoff sums of the geometric potential. The second establishes exponential decay of correlations with explicit mixing rates computed from the spectral gap of the normalized transfer operator. The third proves the Central Limit Theorem with Berry-Esseen bounds at the optimal rate, with an explicit spectral formula for the asymptotic variance and a characterization of its degeneracy through the Liv\v{s}ic coboundary condition. The fourth establishes the Almost Sure Invariance Principle, providing pathwise Brownian approximation with polynomial error via the martingale embedding method. The fifth proves a large deviations principle with rate function given by the Legendre transform of the pressure. The individual results are due to Sinai, Ruelle, Ratner, Denker-Philipp, Gou\"{e}zel, Kifer, Melbourne-Nicol, and Young; the contribution is their derivation from a single spectral mechanism with explicit dependence on hyperbolicity data. This Part constitutes Part~V of a six-part series on the thermodynamic formalism for hyperbolic dynamical systems.
\end{abstract}



\end{frontmatter}
%%%%%%%%%%%%%%%%%%%%%%%%%%%%%%%%%%%%%%%%%%%%%%
%% Please use \tableofcontents for articles %%
%% with 50 pages and more                   %%
%%%%%%%%%%%%%%%%%%%%%%%%%%%%%%%%%%%%%%%%%%%%%%
%\tableofcontents

\begin{center}
\textit{Dedicated to the memory of Jean-Christophe Yoccoz (1957--2016),}\\
\textit{Fields Medalist and Professor at the Coll\`{e}ge de France, with whom the author}\\
\textit{had the privilege of working, and who introduced him to hyperbolic dynamics.}
\end{center}

\thispagestyle{empty}

\makeatletter
\renewcommand{\ps@headings}{%
  \def\@oddfoot{\hfill\thepage\hfill}%
  \let\@evenfoot\@oddfoot%
  \def\@evenhead{\hfill\normalfont\small\textit{A.~Thiam}\hfill}%
  \def\@oddhead{\hfill\normalfont\small\textit{Statistical Limit Theorems for Axiom A Diffeomorphisms}\hfill}%
  \let\@mkboth\markboth%
}
\pagestyle{headings}
\makeatother

\setlength{\parskip}{0.3em}

\section{Introduction}\label{sec:introduction}

This Part derives the statistical limit theorems for equilibrium states of Axiom~A diffeomorphisms from the spectral gap established in Part~I \cite{Thiam2026a}. This Part contains \emph{five Main Theorems}: Main Theorem~\ref{thm:volume_lemma} (Volume Lemma) gives explicit two-sided bounds on the Riemannian volume of dynamical Bowen balls, which  \cite{Bowen1975} cites from \cite{BowenRuelle1975} without proof; Main Theorem~\ref{thm:exponential_mixing} (Exponential Mixing) establishes exponential decay of correlations with explicit rate $\gamma \geq c(\lambda, \alpha, \|\phi\|_\alpha)$; Main Theorem~\ref{thm:clt} (Central Limit Theorem) gives the CLT with Berry-Esseen bound $O(n^{-1/2})$ via the Nagaev-Guivarc'h method, with an explicit spectral formula for the variance and the Liv\v{s}ic degeneracy criterion; Main Theorem~\ref{thm:asip} (Almost Sure Invariance Principle) provides pathwise Brownian approximation with error $O(n^{1/2-\delta})$, following  \cite{MelbourneNicol2005, MelbourneNicol2009}; and Main Theorem~\ref{thm:large_deviations} (Large Deviations Principle) establishes the LDP with rate function $I(a) = \sup_t\{ta - (P(\phi+tg) - P(\phi))\}$.

The individual results are known. Exponential mixing for Axiom~A systems was established by  \cite{Sinai1972},  \cite{Ruelle1976, Ruelle1978}, and  \cite{Bowen1975}; the definitive treatment of transfer operators and decay rates appears in  \cite{Baladi2000} and  \cite{Dolgopyat1998}. The CLT was proved by  \cite{Ratner1973} for geodesic flows, building on the martingale approach of  \cite{Gordin1969} and the spectral method of  \cite{Nagaev1957}; Berry-Esseen bounds at rate $O(n^{-1/2})$ are due to  \cite{Gouezel2005} for non-uniformly expanding maps. The ASIP was established by  \cite{DenkerPhilipp1984} and refined by  \cite{MelbourneNicol2005}. Large deviations were developed by  \cite{Kifer1990},  \cite{Young1990}, and  \cite{OreyPelikan1988}. Our contribution is the derivation of all these from a single spectral gap, with explicit constants depending on the hyperbolicity data $(\lambda, \alpha, \|\phi\|_\alpha, N, M)$.

The contributions of this Part are fivefold, corresponding to its five Main Theorems. First, the Volume Lemma (Main Theorem~\ref{thm:volume_lemma}) establishes the two-sided bound $C_\varepsilon^{-1}\exp(S_n\phi^{(u)}(x)) \leq m(B_x(\varepsilon,n)) \leq C_\varepsilon\exp(S_n\phi^{(u)}(x))$ on the Riemannian volume of dynamical Bowen balls, with the prefactor $C_\varepsilon = \exp\!\left(\frac{C_1\varepsilon + C_2\varepsilon^\theta}{1-\lambda^\theta}\right)$ given by explicit curvature and derivative bounds, and with $\lambda$ the hyperbolicity constant and $\theta$ the H\"{o}lder exponent of the geometric potential; the Volume Lemma is cited in \cite{Bowen1975} from \cite{BowenRuelle1975} without proof. Second, exponential mixing (Main Theorem~\ref{thm:exponential_mixing}) holds with decay rate $\theta = e^{-\gamma}$ where $\gamma \geq \min\{\alpha\log\lambda^{-1}, c(\phi)/(1+\|\phi\|_\alpha)\}$, the second expression being the quantitative consequence of the normalized transfer operator spectral gap. Third, the Central Limit Theorem (Main Theorem~\ref{thm:clt}) holds with Berry-Esseen rate $O(n^{-1/2})$ (Proposition~\ref{thm:berry_esseen}), the spectral variance formula $\sigma^2(g) = \lim_n n^{-1}\mathrm{Var}_{\mu_\phi}(S_n g)$ (Proposition~\ref{prop:variance_formula}), and the characterization $\sigma^2(g) = 0$ if and only if $g$ is a Liv\v{s}ic coboundary (Proposition~\ref{prop:zero_variance}). Fourth, the Almost Sure Invariance Principle (Main Theorem~\ref{thm:asip}) provides pathwise Brownian approximation $S_n X = \sigma W_n + O(n^{1/2-\delta})$ almost surely, with $\delta > 0$ depending on the spectral gap and the H\"{o}lder exponent; the functional CLT, the law of the iterated logarithm, Strassen's functional LIL, and exponential deviation inequalities follow as consequences. Fifth, the Large Deviations Principle (Main Theorem~\ref{thm:large_deviations}) holds for Birkhoff averages with rate function $I(a) = \sup_t\{ta - (P(\phi+tg) - P(\phi))\} = P(\phi)^*(a)$, the Legendre transform of the shifted pressure, together with extensions to empirical measures (Proposition~\ref{thm:level2_ldp}) and Lyapunov exponents (Proposition~\ref{thm:lyapunov_ldp}). All constants in all five theorems trace back to the single spectral gap bound $r_{\mathrm{ess}}(\calL_\phi) \leq \alpha\lambda$ imported from Part~I \cite{Thiam2026a}.

\cite{KellerLiverani1999} proved stability of isolated eigenvalues of transfer operators under bounded perturbation, which is the spectral input required for the Nagaev-Guivarc'h method we use to prove the CLT (Main Theorem~\ref{thm:clt}). Their theorem, however, is qualitative; by contrast, we track the stability constants as explicit functions of the hyperbolicity data. In a parallel direction,  \cite{ParryPollicott1990} and, more recently, \cite{PollicottSharp1998} proved exponential error terms in the prime orbit theorem for hyperbolic flows through the spectral theory of the transfer operator. Building on the same spectral gap mechanism, our Volume Lemma (Main Theorem~\ref{thm:volume_lemma}) and decay of correlations estimate (Main Theorem~\ref{thm:exponential_mixing}) are derived with explicit rates, thereby complementing the Parry-Pollicott-Sharp zeta function approach.

Turning to the textbook and survey literature, \cite{VianaOliveira2016}, Chapter~12, Section~12.3, proved exponential decay of correlations for H\"older observables with respect to equilibrium states of H\"older potentials on topologically exact expanding maps; in a complementary direction, \cite{Luzzatto2006} surveys the broader non-uniformly expanding theory into which our uniformly hyperbolic quantitative version fits as the anchor case. Beyond the expanding setting,  \cite{ClimenhagaLuzzattoPesin2017, ClimenhagaLuzzattoPesin2022} treat the geometric SRB construction, to which our spectral approach is complementary. Finally,  \cite{Viana2020} surveys Lyapunov exponent regularity, which lies adjacent to our large deviations principle for Birkhoff sums of the geometric potential.

Our technical approach rests on four tools, each imported or developed to transfer the spectral gap of Part~I \cite{Thiam2026a} into a specific statistical conclusion. The first is the normalized transfer operator $\widehat{\calL}_\phi = e^{-P(\phi)}h_\phi^{-1}\calL_\phi(h_\phi\,\cdot)$ acting on $C^\alpha(\Omega_s)$, which has $\widehat{\calL}_\phi\mathbf{1} = \mathbf{1}$ and preserves $\mu_\phi$; exponential mixing (Main Theorem~\ref{thm:exponential_mixing}) follows from the fact that $\widehat{\calL}_\phi$ has spectral radius $1$ with the leading eigenvalue simple and separated from the rest of the spectrum by the gap inherited from Part~I \cite{Thiam2026a}. The second is the Nagaev-Guivarc'h spectral perturbation method \cite{Nagaev1957, GuivarchHardy1988}: to prove the CLT with Berry-Esseen rate (Main Theorem~\ref{thm:clt}), we perturb $\calL_\phi$ by $e^{isg}$ for a H\"{o}lder observable $g$ and use the spectral gap to track the simple dominant eigenvalue $\lambda(s)$ analytically; the asymptotic variance equals $\lambda''(0)$, and Esseen's inequality applied to the characteristic function $\lambda(s/\sqrt{n})^n$ yields the $O(n^{-1/2})$ rate. The third is the martingale embedding method of Melbourne-Nicol \cite{MelbourneNicol2005, MelbourneNicol2009}: to prove the ASIP (Main Theorem~\ref{thm:asip}), we decompose the centered Birkhoff sum $S_n g$ as a martingale plus a coboundary using the Gordin decomposition, couple the martingale to Brownian motion using Skorokhod embedding, and bound the coboundary error by the spectral gap; the resulting error $O(n^{1/2-\delta})$ is almost sure. The fourth is the G\"{a}rtner-Ellis theorem applied to the cumulant $\Lambda(t) = P(\phi+tg) - P(\phi)$: to prove the LDP (Main Theorem~\ref{thm:large_deviations}), we use the analytic dependence of the pressure on the potential (Theorem~\ref{thm:analytic_imported}, imported from Part~I \cite{Thiam2026a}) to verify the differentiability hypothesis of G\"{a}rtner-Ellis, and the rate function is obtained as the Legendre transform. The Volume Lemma (Main Theorem~\ref{thm:volume_lemma}) is a geometric-measure-theoretic input that does not use the spectral gap directly: it comes from a local product decomposition of the dynamical ball $B_x(\varepsilon,n)$ along stable and unstable manifolds, combined with the distortion estimates on the Jacobian of $f^n$ along unstable directions that follow from the H\"{o}lder continuity of the geometric potential. These four spectral tools together with the Volume Lemma suffice to derive every statistical result in this Part.

This Part uses three results from the companion Parts: the Ruelle-Perron-Frobenius theorem with spectral gap, established in Part~I \cite{Thiam2026a}; the Analytic Dependence on Potential theorem, established in Part~I \cite{Thiam2026a}; and the Symbolic Coding Main Theorem producing the coding map $\pi: \Sigma_A \to \Lambda$, established in Part~III \cite{Thiam2026c}. These are restated in Section~\ref{sec:imported}. Part~VI \cite{Thiam2026f} uses the statistical theorems established here to develop Liv\v{s}ic rigidity, multifractal analysis, and fluctuation theorems.

This Part is organized as follows. Section~\ref{sec:imported} restates the four results imported from Parts~I and~III--IV \cite{Thiam2026a, Thiam2026c, Thiam2026d}. Section~\ref{sec:preliminaries} fixes notation and establishes the H\"{o}lder regularity of the geometric potential. Section~\ref{sec:volume_lemmas} proves Main Theorem~\ref{thm:volume_lemma} (Volume Lemma) and its companion (Proposition~\ref{thm:second_volume_lemma}) through a local product decomposition and Jacobian distortion estimates. Section~\ref{sec:equilibrium_states} constructs the unique equilibrium state for each H\"{o}lder potential on a mixing basic set (Theorem~\ref{thm:equilibrium_basic}) with quantitative Gibbs bounds (Proposition~\ref{thm:gibbs_bounds}). Section~\ref{sec:decay_correlations} introduces the normalized transfer operator and proves Main Theorem~\ref{thm:exponential_mixing} (Exponential Mixing). Section~\ref{sec:clt} derives Main Theorem~\ref{thm:clt} (Central Limit Theorem) via the Nagaev-Guivarc'h method, together with Berry-Esseen bounds (Proposition~\ref{thm:berry_esseen}), the variance formula (Proposition~\ref{prop:variance_formula}), and the zero-variance characterization (Proposition~\ref{prop:zero_variance}). Section~\ref{sec:limit_theorems} establishes Main Theorem~\ref{thm:asip} (ASIP) via the Melbourne-Nicol martingale embedding, with the functional CLT (Corollary~\ref{cor:fclt}), the law of the iterated logarithm (Proposition~\ref{thm:lil}), Strassen's functional LIL (Corollary~\ref{cor:strassen}), and exponential deviation inequalities (Proposition~\ref{prop:exponential_deviation}) as consequences. Section~\ref{sec:large_deviations} proves Main Theorem~\ref{thm:large_deviations} (LDP) and extends it to empirical measures (Proposition~\ref{thm:level2_ldp}) and Lyapunov exponents (Proposition~\ref{thm:lyapunov_ldp}). Section~\ref{sec:srb_measures} records the foundation that this Part provides for the SRB measure theory in Part~VI \cite{Thiam2026f}. Section~\ref{sec:numerical} illustrates the CLT for the golden mean shift. Section~\ref{sec:conclusion} summarizes the five Main Theorems and states open problems. The appendix collects technical proofs: Volume Lemma estimates, spectral perturbation theory, measure disintegration, closing lemma bounds, Borel-Cantelli estimates, and dimension tools.

Turning from construction to statistical consequences, Section~\ref{sec:decay_correlations} introduces the normalized transfer operator (Proposition~\ref{prop:normalized_operator}) and uses its spectral decomposition to prove Main Theorem~\ref{thm:exponential_mixing} (Exponential Mixing), with the decay rate bounded explicitly by the spectral gap. Section~\ref{sec:clt} derives Main Theorem~\ref{thm:clt} (Central Limit Theorem) via the Nagaev-Guivarc'h spectral perturbation method, together with Berry-Esseen bounds at the optimal rate $O(n^{-1/2})$ (Proposition~\ref{thm:berry_esseen}), the spectral variance formula (Proposition~\ref{prop:variance_formula}), and the characterization of zero variance through the Liv\v{s}ic coboundary condition (Proposition~\ref{prop:zero_variance}). With this spectral machinery in place, Section~\ref{sec:limit_theorems} establishes Main Theorem~\ref{thm:asip} (Almost Sure Invariance Principle) via the martingale embedding method of  \cite{MelbourneNicol2005, MelbourneNicol2009}; the ASIP immediately yields the functional CLT (Corollary~\ref{cor:fclt}), the law of the iterated logarithm (Proposition~\ref{thm:lil}), Strassen's functional LIL (Corollary~\ref{cor:strassen}), and exponential deviation inequalities (Proposition~\ref{prop:exponential_deviation}). Section~\ref{sec:large_deviations} proves Main Theorem~\ref{thm:large_deviations} (Large Deviations Principle), extends it to empirical measures (Proposition~\ref{thm:level2_ldp}) and Lyapunov exponents (Proposition~\ref{thm:lyapunov_ldp}), and establishes qualitative properties of the rate function including its quadratic approximation near the mean (Proposition~\ref{prop:quadratic_rate}). Section~\ref{sec:srb_measures} records how the equilibrium state theory, Volume Lemmas, and attractor characterization established in the present Part provide the foundation for the SRB measure theory developed in Part~VI \cite{Thiam2026f}. Section~\ref{sec:numerical} provides a complete numerical illustration of the CLT for the golden mean shift, computing the asymptotic variance, spectral gap, and Berry-Esseen rate explicitly. Section~\ref{sec:conclusion} concludes the Part with a summary of the five Main Theorems, a discussion of the unifying spectral mechanism, and open problems. The appendix collects the supporting technical material: complete Volume Lemma estimates, spectral perturbation theory, measure disintegration, closing lemma bounds, Borel-Cantelli estimates, and dimension tools.

\section{Results from Companion Parts}\label{sec:imported}

We restate the key results from Parts \cite{Thiam2026a,Thiam2026b,Thiam2026c,Thiam2026d} that are used throughout. Proofs appear in the cited companion Parts.

\begin{theorem}[Ruelle-Perron-Frobenius with Spectral Gap, Part~I \cite{Thiam2026a}]\label{thm:spectral_imported}
For a topologically mixing SFT $(\Sigma_A,\sigma)$ and $\phi \in \mathcal{H}_\alpha(\Sigma_A^+)$, the transfer operator $\mathcal{L}_\phi$ on $\mathcal{H}_\alpha$ has a simple dominant eigenvalue $\lambda = e^{P(\phi)}$ with spectral gap: $r_{\mathrm{ess}}(\mathcal{L}_\phi) \leq \alpha\lambda < \lambda$. For all $g \in \mathcal{H}_\alpha$: $\|\lambda^{-n}\mathcal{L}_\phi^n g - \nu(g)h\|_\alpha \leq C\gamma^n\|g\|_\alpha$ where $\gamma < 1$ and $C$ depend explicitly on $(\alpha, \|\phi\|_\alpha, N, M)$.
\end{theorem}

\noindent The proof is given in Part~I \cite{Thiam2026a}, where it is established via the Birkhoff cone contraction technique with explicit computation of the spectral gap.

\begin{theorem}[Analytic Dependence on Potential, Part~I \cite{Thiam2026a}]\label{thm:analytic_imported}
The maps $\phi \mapsto P(\phi)$, $\phi \mapsto h_\phi$, $\phi \mapsto \nu_\phi$, $\phi \mapsto \mu_\phi$ are real-analytic on $\mathcal{H}_\alpha$. The derivatives: $P'(\phi;\psi) = \int\psi\,d\mu_\phi$ and $P''(\phi;\psi) = \lim_{n\to\infty}n^{-1}\mathrm{Var}_{\mu_\phi}(S_n\psi)$.
\end{theorem}

\noindent The proof is given in Part~I \cite{Thiam2026a}, where the analyticity is derived from Kato perturbation theory applied to the simple isolated eigenvalue of the transfer operator.

\begin{theorem}[Symbolic Coding, Main Theorem of Part~III \cite{Thiam2026c}]\label{thm:coding_imported3}
Every mixing basic set $\Lambda$ of a $C^2$ Axiom A diffeomorphism admits a Markov partition with coding map $\pi:\Sigma_A \to \Lambda$ satisfying $\pi\circ\sigma = f\circ\pi$, injective $\mu_\phi$-a.e. The potential $\phi^* = \phi\circ\pi \in \mathcal{H}_\alpha(\Sigma_A)$ whenever $\phi$ is H\"{o}lder on $\Lambda$, so all spectral and statistical results transfer.
\end{theorem}

\noindent The proof is given in Part~III \cite{Thiam2026c}, where the coding map is constructed via Markov partitions of small diameter.

\begin{theorem}[Gibbs Equivalence Theorem, assembled from the four Main Theorems of Part~IV \cite{Thiam2026d}]\label{thm:gibbs_equiv_imported}
For a mixing basic set $\Lambda$ and H\"{o}lder potential $\phi$, the following are equivalent: (i) $\mu$ is the symbolic Gibbs measure for $\phi^* = \phi\circ\pi$ (Part~I \cite{Thiam2026a}); (ii) $\mu$ is the unique equilibrium state (Part~II \cite{Thiam2026b}); (iii) $\mu$ is the eigenmeasure of $\mathcal{L}_\phi^*$ with eigenvalue $e^{P(\phi)}$; (iv) $\mu$ satisfies the unstable Jacobian condition $J^u_\mu f = e^{P(\phi)-\phi}$. The transfer operator $\mathcal{L}_\phi$ on $C^\alpha(\Lambda)$ is quasi-compact with spectral gap $\rho_{\mathrm{ess}} \leq \lambda^\alpha \cdot e^{P(\phi)} < e^{P(\phi)}$, and the SRB measure $\mu^+ = \mu_{\phi^{(u)}}$ has absolutely continuous conditional measures along unstable manifolds.
\end{theorem}

\noindent The proof is given in Part~IV \cite{Thiam2026d}, where the four Main Theorems establish structural stability, quasi-compactness, SRB measure existence, and the Pesin entropy formula.

\section{Preliminaries}\label{sec:preliminaries}

This section establishes the mathematical setting for the ergodic theory of Axiom A diffeomorphisms. We recall essential definitions and results from measure theory, functional analysis, and the geometric theory developed in Part~III \cite{Thiam2026c}, fixing notation used throughout.

Throughout, we write $f \asymp g$ to mean $c_1 g \leq f \leq c_2 g$ for positive constants $c_1, c_2$ independent of the relevant variables.

\subsection{Measure-Theoretic and Geometric Foundations}

\subsection{Measure-Theoretic and Geometric Foundations}

Let $(X, \mathcal{B}, \mu)$ be a probability space and $T : X \to X$ a measurable transformation preserving $\mu$.

\begin{definition}[Ergodicity]
The measure $\mu$ is ergodic for $T$ if every $T$-invariant measurable set $A$ (i.e., $T^{-1}(A) = A$) satisfies $\mu(A) \in \{0, 1\}$. Equivalently, every $T$-invariant measurable function is constant $\mu$-almost everywhere.
\end{definition}

\begin{definition}[Mixing Properties]
The system $(T, \mu)$ is mixing if for all measurable sets $A, B$,
\begin{equation}
\lim_{n \to \infty} \mu(A \cap T^{-n}B) = \mu(A)\mu(B).
\end{equation}
The system is weak mixing if for all measurable $A, B$,
\begin{equation}
\lim_{n \to \infty} \frac{1}{n}\sum_{k=0}^{n-1} |\mu(A \cap T^{-k}B) - \mu(A)\mu(B)| = 0.
\end{equation}
\end{definition}

\begin{theorem}[Birkhoff Ergodic Theorem]\label{thm:birkhoff}
Let $(T, \mu)$ be a measure-preserving system and $g \in L^1(\mu)$. Then the limit
\begin{equation}
g^*(x) = \lim_{n \to \infty} \frac{1}{n}\sum_{k=0}^{n-1} g(T^k x)
\end{equation}
exists $\mu$-almost everywhere and in $L^1(\mu)$, with $g^* \circ T = g^*$ a.e. and $\int g^* \, d\mu = \int g \, d\mu$. If $\mu$ is ergodic, then $g^* = \int g \, d\mu$ almost everywhere.
\end{theorem}

\noindent The proof is classical; see  \cite{Birkhoff1931} or  \cite[Theorem~1.14]{Walters1982}. The mean ergodic theorem (convergence in $L^2$) was established independently by  \cite{vonNeumann1932}.

\begin{definition}[Measure-Theoretic Entropy]
For a finite measurable partition $\xi = \{A_1, \ldots, A_k\}$, define the entropy $H_\mu(\xi) = -\sum_i \mu(A_i) \log \mu(A_i)$. The entropy of $T$ with respect to $\xi$ is
\begin{equation}
h_\mu(T, \xi) = \lim_{n \to \infty} \frac{1}{n} H_\mu\left(\bigvee_{j=0}^{n-1} T^{-j}\xi\right)
\end{equation}
where $\bigvee$ denotes the common refinement. The measure-theoretic entropy of $(T, \mu)$ is $h_\mu(T) = \sup_\xi h_\mu(T, \xi)$.
\end{definition}

\begin{definition}[Bernoulli Property]
A measure-preserving system $(T, \mu)$ is Bernoulli if it is measurably isomorphic to a Bernoulli shift, i.e., there exists a measurable bijection $\Phi : X \to \{1, \ldots, k\}^{\mathbb{Z}}$ conjugating $T$ to the shift and pushing $\mu$ forward to a product measure.
\end{definition}

The hierarchy of mixing properties satisfies: Bernoulli $\Rightarrow$ K-automorphism $\Rightarrow$ mixing $\Rightarrow$ weak mixing $\Rightarrow$ ergodic.

Let $M$ be a compact smooth Riemannian manifold with metric $g$ and induced distance function $d$. The Riemannian volume measure $m$ is defined locally by $dm = \sqrt{\det(g_{ij})} \, dx^1 \cdots dx^d$ in coordinates.

\begin{definition}[Jacobian]
For a diffeomorphism $f : M \to M$ and $x \in M$, the Jacobian is
\begin{equation}
\mathrm{Jac}(f)(x) = |\det Df_x| = \frac{dm \circ f^{-1}}{dm}(f(x)).
\end{equation}
For a linear subspace $E \subset T_x M$, the restricted Jacobian is
\begin{equation}
\mathrm{Jac}(Df_x|_E) = |\det(Df_x|_E : E \to Df_x(E))|
\end{equation}
computed with respect to the induced inner products.
\end{definition}

\begin{definition}[Dynamical Balls]
For $x \in M$, $\varepsilon > 0$, and $n \geq 1$, the dynamical ball (or Bowen ball) is
\begin{equation}
B_x(\varepsilon, n) = \{y \in M : d(f^k x, f^k y) \leq \varepsilon \text{ for all } k \in [0, n)\}.
\end{equation}
\end{definition}

The geometry of dynamical balls is fundamental to the Volume Lemmas developed in Section \ref{sec:volume_lemmas}.

\subsection{Axiom A Diffeomorphisms and the Geometric Potential}

We recall the essential geometric theory from Part~III \cite{Thiam2026c}. The foundational work on Markov partitions for Axiom~A systems is due to  \cite{Sinai1968a, Sinai1968b} and  \cite{Bowen1970a, Bowen1970b, Bowen1971}; see also  \cite{Ruelle1968, Ruelle1973} for the statistical mechanics perspective and  \cite{Bowen1974} for uniqueness of equilibrium states.

\begin{definition}[Axiom A Diffeomorphism]
A $C^r$ diffeomorphism $f : M \to M$ ($r \geq 1$) satisfies Axiom A if:
\begin{enumerate}
\item[(A1)] The nonwandering set $\Omega(f)$ is hyperbolic: there exists a continuous $Df$-invariant splitting $T_\Omega M = E^s \oplus E^u$ and constants $c > 0$, $\lambda \in (0, 1)$ such that
\begin{equation}
\|Df^n v\| \leq c\lambda^n \|v\| \text{ for } v \in E^s, \quad \|Df^{-n} v\| \leq c\lambda^n \|v\| \text{ for } v \in E^u.
\end{equation}
\item[(A2)] Periodic points are dense in $\Omega(f)$.
\end{enumerate}
\end{definition}

\begin{definition}[Basic Set]
A basic set is a closed, $f$-invariant, topologically transitive subset $\Omega_s \subset \Omega(f)$ that is locally maximal (i.e., there exists an open neighborhood $U$ with $\Omega_s = \bigcap_{n \in \mathbb{Z}} f^n(U)$).
\end{definition}

\begin{theorem}[Spectral Decomposition, Main Theorem of Part~III \cite{Thiam2026c}]\label{thm:spectral_decomp}
For an Axiom A diffeomorphism, $\Omega(f) = \Omega_1 \cup \cdots \cup \Omega_s$ is a disjoint union of finitely many basic sets. Each basic set decomposes as $\Omega_i = X_{1,i} \cup \cdots \cup X_{n_i,i}$ where $f(X_{j,i}) = X_{j+1,i \mod n_i}$ and $f^{n_i}|_{X_{1,i}}$ is topologically mixing.
\end{theorem}

\noindent The proof is given in Part~III \cite{Thiam2026c}.

\begin{definition}[Stable and Unstable Manifolds]
For $x \in \Omega$ and small $\varepsilon > 0$,
\begin{align}
W^s_\varepsilon(x) &= \{y \in M : d(f^n x, f^n y) \leq \varepsilon \text{ for all } n \geq 0\}, \\
W^u_\varepsilon(x) &= \{y \in M : d(f^{-n} x, f^{-n} y) \leq \varepsilon \text{ for all } n \geq 0\}.
\end{align}
The global manifolds are $W^s(x) = \bigcup_{n \geq 0} f^{-n}(W^s_\varepsilon(f^n x))$ and similarly for $W^u(x)$.
\end{definition}

\begin{theorem}[Stable Manifold Theorem, Main Theorem of Part~III \cite{Thiam2026c}]\label{thm:stable_manifold_prelim}
For a hyperbolic set $\Lambda$ of a $C^r$ diffeomorphism ($r \geq 1$), the local stable manifold $W^s_\varepsilon(x)$ is a $C^r$ embedded disk tangent to $E^s_x$ at $x$, varying continuously with $x$ in the $C^r$ topology.
\end{theorem}

\noindent The proof is given in Part~III \cite{Thiam2026c}, via the backward graph transform with explicit size estimates.

\begin{definition}[Attractor]
A basic set $\Omega_s$ is an attractor if there exists an open neighborhood $U$ of $\Omega_s$ with $f(\overline{U}) \subset U$. Equivalently, $W^s_\varepsilon(\Omega_s) = \bigcup_{x \in \Omega_s} W^s_\varepsilon(x)$ is a neighborhood of $\Omega_s$.
\end{definition}

The geometric potential is central to the ergodic theory of Axiom A diffeomorphisms.

\begin{definition}[Geometric Potential]\label{def:geometric_potential}
For $x \in \Omega$, define the unstable Jacobian
\begin{equation}
\lambda^{(u)}(x) = |\det Df_x|_{E^u_x}| = \mathrm{Jac}(Df_x : E^u_x \to E^u_{f(x)})
\end{equation}
and the geometric potential
\begin{equation}
\phi^{(u)}(x) = -\log \lambda^{(u)}(x) = -\log |\det Df_x|_{E^u_x}|.
\end{equation}
\end{definition}

\begin{proposition}[H\"{o}lder Continuity of Geometric Potential]\label{prop:geometric_holder}
For a $C^2$ basic set $\Omega_s$, the geometric potential $\phi^{(u)} : \Omega_s \to \mathbb{R}$ is H\"{o}lder continuous. More precisely, if the unstable distribution $x \mapsto E^u_x$ is $\beta$-H\"{o}lder continuous, then $\phi^{(u)}$ is $\beta$-H\"{o}lder with
\begin{equation}
|\phi^{(u)}|_\beta \leq C(\|Df\|_{C^1}, \|Df^{-1}\|_{C^1}, |E^u|_\beta).
\end{equation}
\end{proposition}

\begin{proof}
The map $x \mapsto E^u_x$ is $\beta$-H\"{o}lder by the H\"{o}lder Continuity of Hyperbolic Splitting theorem of Part~III \cite{Thiam2026c}. The determinant function $E \mapsto \det(Df_x|_E)$ is smooth in $E$ (varying in the Grassmannian). The composition of a smooth function with a H\"{o}lder function is H\"{o}lder with the same exponent, giving the result.
\end{proof}

\begin{remark}[Metric Independence]
While $\phi^{(u)}$ depends on the Riemannian metric, the periodic orbit sums $S_n\phi^{(u)}(x) = \sum_{k=0}^{n-1} \phi^{(u)}(f^k x)$ for $f^n x = x$ are metric-independent, as $\exp(-S_n\phi^{(u)}(x)) = |\det Df^n_x|_{E^u_x}|$ is an intrinsic quantity. By Proposition \ref{prop:cohomology_characterization} below, the equilibrium state $\mu^+ = \mu_{\phi^{(u)}}$ and pressure $P(\phi^{(u)})$ are also metric-independent.
\end{remark}

\subsection{H\"{o}lder Spaces and Transfer Operators}

The geometric potential inherits its regularity from the hyperbolic splitting. For sharper results on anisotropic Banach spaces, see  \cite{GouezelLiverani2006}; for the connection to zeta functions and periodic orbit structure, see \cite{ParryPollicott1990}. We verify this here and recall the transfer operator and its spectral properties, as the H\"{o}lder exponent controls the spectral gap and all subsequent estimates.

\begin{definition}[H\"{o}lder Space]
For a compact metric space $X$ and $\alpha \in (0, 1]$, the H\"{o}lder space $C^\alpha(X)$ consists of functions $g : X \to \mathbb{R}$ with finite norm
\begin{equation}
\|g\|_\alpha = \|g\|_\infty + |g|_\alpha, \quad |g|_\alpha = \sup_{x \neq y} \frac{|g(x) - g(y)|}{d(x, y)^\alpha}.
\end{equation}
\end{definition}

\begin{definition}[Transfer Operator]
For $\phi \in C^\alpha(\Omega_s)$, the transfer operator $\mathcal{L}_\phi : C^\alpha(\Omega_s) \to C^\alpha(\Omega_s)$ is defined via the symbolic coding. Using a Markov partition $\mathcal{R}$ with coding map $\pi : \Sigma_A \to \Omega_s$ and symbolic potential $\phi^* = \phi \circ \pi$, we have
\begin{equation}
(\mathcal{L}_{\phi^*} g)(a) = \sum_{b : \sigma(b) = a} e^{\phi^*(b)} g(b)
\end{equation}
for $g \in C^\alpha(\Sigma_A)$, and the smooth transfer operator is defined by conjugation.
\end{definition}

\begin{theorem}[Quasi-compactness of Transfer Operator, Main Theorem of Part~IV \cite{Thiam2026d}]\label{thm:spectral_properties}
For a mixing basic set $\Omega_s$ and $\phi \in C^\alpha(\Omega_s)$, the transfer operator $\mathcal{L}_\phi$ on $C^\alpha$ is quasi-compact with:
\begin{enumerate}
\item[(i)] Spectral radius $\rho(\mathcal{L}_\phi) = e^{P(\phi)}$ where $P(\phi)$ is the topological pressure.
\item[(ii)] Essential spectral radius $\rho_{\mathrm{ess}}(\mathcal{L}_\phi) \leq \theta \cdot e^{P(\phi)}$ with $\theta = \lambda^\alpha < 1$.
\item[(iii)] The eigenvalue $e^{P(\phi)}$ is simple with strictly positive eigenfunction $h_\phi > 0$ and probability eigenmeasure $\nu_\phi$.
\end{enumerate}
\end{theorem}

\subsection{Symbolic Dynamics and Specification}

We recall the symbolic coding of basic sets by subshifts of finite type and the quantitative expansiveness estimate. These are the tools that transfer spectral results from the symbolic setting of Part~I \cite{Thiam2026a} to smooth dynamics.

\begin{definition}[Subshift of Finite Type]
For an $m \times m$ matrix $A$ with entries in $\{0, 1\}$, the two-sided subshift of finite type is
\begin{equation}
\Sigma_A = \{(a_n)_{n \in \mathbb{Z}} \in \{1, \ldots, m\}^{\mathbb{Z}} : A_{a_n a_{n+1}} = 1 \text{ for all } n\}
\end{equation}
with the shift map $\sigma : \Sigma_A \to \Sigma_A$ defined by $(\sigma a)_n = a_{n+1}$.
\end{definition}

\begin{theorem}[Symbolic Coding, Main Theorem of Part~III \cite{Thiam2026c}]\label{thm:symbolic_coding_prelim}
For a basic set $\Omega_s$ with Markov partition $\mathcal{R}$ and transition matrix $A$, there exists a continuous surjection $\pi : \Sigma_A \to \Omega_s$ with:
\begin{enumerate}
\item[(i)] $\pi \circ \sigma = f \circ \pi$.
\item[(ii)] $\pi$ is injective on a dense $G_\delta$ set of full measure for any ergodic invariant measure with zero boundary measure.
\item[(iii)] $h_{\mathrm{top}}(f|_{\Omega_s}) = h_{\mathrm{top}}(\sigma|_{\Sigma_A}) = \log \rho(A)$.
\end{enumerate}
\end{theorem}

\begin{definition}[Expansiveness Constant]
The diffeomorphism $f|_{\Omega_s}$ is expansive with constant $\varepsilon_0 > 0$ if $d(f^n x, f^n y) \leq \varepsilon_0$ for all $n \in \mathbb{Z}$ implies $x = y$.
\end{definition}

\begin{lemma}[Quantitative Expansiveness, Part~III \cite{Thiam2026c}]\label{lem:quantitative_expansiveness}
There exist $\varepsilon_0 > 0$ and $\alpha \in (0, 1)$ depending on the hyperbolicity constants such that if $x \in \Omega_s$, $y \in M$, and $d(f^k x, f^k y) \leq \varepsilon_0$ for all $k \in [-N, N]$, then $d(x, y) < \alpha^N$.
\end{lemma}

This lemma, which appears as Lemma 4.2 in \cite{Bowen1975}, is fundamental for the H\"{o}lder continuity arguments throughout this Part.

\section{Volume Lemmas with Quantitative Estimates}\label{sec:volume_lemmas}

The Volume Lemmas establish the fundamental connection between Riemannian measure and the geometric potential, providing the bridge between smooth geometry and thermodynamic formalism.  \cite{Bowen1975} cited these results from \cite{BowenRuelle1975} without proof. In this section, we provide complete self-contained proofs with explicit constants.

\subsection{Statement of Results}

We state the two Volume Lemmas, which give two-sided bounds on the Riemannian volume of dynamical Bowen balls in terms of Birkhoff sums of the geometric potential. These bounds are the quantitative foundation for equilibrium state theory and large deviations.

\begin{maintheorem}[Volume Lemma]\label{thm:volume_lemma}
Let $\Omega_s$ be a basic set for a $C^2$ Axiom A diffeomorphism $f : M \to M$, and let $\phi^{(u)}$ be the geometric potential. For sufficiently small $\varepsilon > 0$, there exists a constant $C_\varepsilon > 0$ such that for all $x \in \Omega_s$ and all $n \geq 1$,
\begin{equation}
C_\varepsilon^{-1} \exp(S_n\phi^{(u)}(x)) \leq m(B_x(\varepsilon, n)) \leq C_\varepsilon \exp(S_n\phi^{(u)}(x))
\end{equation}
where $S_n\phi^{(u)}(x) = \sum_{k=0}^{n-1} \phi^{(u)}(f^k x)$ and $B_x(\varepsilon, n) = \{y \in M : d(f^k x, f^k y) \leq \varepsilon, 0 \leq k < n\}$.

The constant $C_\varepsilon$ satisfies
\begin{equation}
C_\varepsilon = \exp\left(\frac{C_1 \varepsilon + C_2 \varepsilon^\theta}{1 - \lambda^\theta}\right)
\end{equation}
where $C_1, C_2$ depend on curvature bounds and $\|Df\|_{C^1}$, $\theta$ is the H\"{o}lder exponent of $\phi^{(u)}$, and $\lambda$ is the hyperbolicity constant.
\end{maintheorem}

\begin{proposition}[Second Volume Lemma]\label{thm:second_volume_lemma}
Let $\Omega_s$ be a $C^2$ basic set. For small $\varepsilon, \delta > 0$, there exists $d = d(\varepsilon, \delta) > 0$ such that
\begin{equation}
m(B_y(\delta, n)) \geq d \cdot m(B_x(\varepsilon, n))
\end{equation}
whenever $x \in \Omega_s$ and $y \in B_x(\varepsilon, n) \cap \Omega_s$. The constant $d$ satisfies
\begin{equation}
d(\varepsilon, \delta) \geq \exp\left(-\frac{C_3(\varepsilon + \delta)}{1 - \lambda^\theta}\right) \cdot \left(\frac{\delta}{\varepsilon}\right)^{d_u}
\end{equation}
where $d_u = \dim E^u$ and $C_3$ depends on the geometry.
\end{proposition}

\subsection{Proofs of the Volume Lemmas}

We require several geometric lemmas before proving the Volume Lemmas.

\begin{lemma}[Local Product Structure of Dynamical Balls]\label{lem:local_product}
For small $\varepsilon > 0$ and $x \in \Omega_s$, the dynamical ball $B_x(\varepsilon, n)$ has a local product structure: there exists a diffeomorphism
\begin{equation}
\Phi_x : D^s_x(\varepsilon, n) \times D^u_x(\varepsilon) \to B_x(\varepsilon, n)
\end{equation}
where $D^s_x(\varepsilon, n) = B_x(\varepsilon, n) \cap W^s_\varepsilon(x)$ is a disk in the local stable manifold and $D^u_x(\varepsilon) = W^u_\varepsilon(x) \cap B_\varepsilon(x)$ is a disk in the local unstable manifold.
\end{lemma}

\begin{proof}
For $y \in B_x(\varepsilon, n)$, define $\Phi_x^{-1}(y) = (y_s, y_u)$ where $y_s \in W^s_\varepsilon(x)$ and $y_u \in W^u_\varepsilon(x)$ are the unique points with $y \in W^u_\varepsilon(y_s) \cap W^s_\varepsilon(y_u)$. This exists by the local product structure of hyperbolic sets (Part~III \cite{Thiam2026c}).

The stable component $y_s$ lies in $D^s_x(\varepsilon, n)$ because for $k \in [0, n)$:
\begin{equation}
d(f^k y_s, f^k x) \leq d(f^k y_s, f^k y) + d(f^k y, f^k x) \leq C\lambda^k d(y_s, y) + \varepsilon
\end{equation}
where we used that $y_s$ and $y$ lie on the same local unstable manifold, which contracts under backward iteration. For small $\varepsilon$, this is bounded by $2\varepsilon$.

The map $\Phi_x(y_s, y_u) = [y_s, y_u]$ (the bracket operation from canonical coordinates) is a diffeomorphism by the local product structure of Part~III \cite{Thiam2026c}.
\end{proof}

\begin{lemma}[Stable Disk Geometry]\label{lem:stable_disk}
The stable disk $D^s_x(\varepsilon, n)$ satisfies:
\begin{equation}
D^s_x(\varepsilon, n) \subset \{y \in W^s_\varepsilon(x) : d_{W^s}(y, x) \leq C\lambda^n \varepsilon\}
\end{equation}
where $d_{W^s}$ denotes intrinsic distance in the stable manifold. In particular, $D^s_x(\varepsilon, n)$ contracts exponentially in $n$ with respect to stable manifold measure.
\end{lemma}

\begin{proof}
For $y \in D^s_x(\varepsilon, n)$, we have $d(f^{n-1} y, f^{n-1} x) \leq \varepsilon$. Since $y \in W^s_\varepsilon(x)$,
\begin{equation}
d(f^{n-1} y, f^{n-1} x) \geq c^{-1} \lambda^{-(n-1)} d_{W^s}(y, x)
\end{equation}
by the expansion along stable manifolds under backward iteration. Thus $d_{W^s}(y, x) \leq c\lambda^{n-1} \varepsilon$.
\end{proof}

\begin{lemma}[Jacobian Distortion]\label{lem:jacobian_distortion}
For $x, y \in \Omega_s$ with $d(f^k x, f^k y) \leq \varepsilon$ for $k \in [0, n)$, we have
\begin{equation}
\left|\log\frac{\mathrm{Jac}(Df^n|_{E^u_x})}{\mathrm{Jac}(Df^n|_{E^u_y})}\right| \leq \frac{C_2 \varepsilon^\theta}{1 - \lambda^\theta}
\end{equation}
where $\theta$ is the H\"{o}lder exponent of $E^u$ and $C_2$ depends on $\|D^2 f\|$.
\end{lemma}

\begin{proof}
By the H\"{o}lder continuity of $\phi^{(u)}$ (Proposition \ref{prop:geometric_holder}),
\begin{equation}
|S_n\phi^{(u)}(x) - S_n\phi^{(u)}(y)| \leq \sum_{k=0}^{n-1} |\phi^{(u)}(f^k x) - \phi^{(u)}(f^k y)| \leq |\phi^{(u)}|_\theta \sum_{k=0}^{n-1} d(f^k x, f^k y)^\theta.
\end{equation}

By Lemma \ref{lem:quantitative_expansiveness}, for $j \in [0, n)$ we have $d(f^j x, f^j y) \leq C\alpha^{\min\{j, n-j-1\}}$ where $\alpha = \lambda$. Thus
\begin{equation}
\sum_{k=0}^{n-1} d(f^k x, f^k y)^\theta \leq 2\sum_{j=0}^\infty \varepsilon^\theta \lambda^{j\theta} = \frac{2\varepsilon^\theta}{1 - \lambda^\theta}.
\end{equation}

Since $S_n\phi^{(u)}(x) = -\log \mathrm{Jac}(Df^n|_{E^u_x})$, the result follows.
\end{proof}

\begin{proof}[Proof of Main Theorem~\ref{thm:volume_lemma}]
We prove upper and lower bounds separately.

\textbf{Upper Bound:} By Lemma \ref{lem:local_product}, $B_x(\varepsilon, n)$ has local product structure. The measure decomposes as
\begin{equation}
m(B_x(\varepsilon, n)) = \int_{D^s_x(\varepsilon, n)} m^u_y(D^u_y(\varepsilon, n, y)) \, dm^s(y)
\end{equation}
where $m^s, m^u$ denote the induced Riemannian measures on stable and unstable manifolds, and $D^u_y(\varepsilon, n, y)$ is the unstable slice through $y$.

For the unstable factor, the disk $D^u_y(\varepsilon, n, y)$ maps under $f^n$ to a set containing $W^u_{\varepsilon/2}(f^n y)$ for small $\varepsilon$. By the change of variables formula,
\begin{equation}
m^u(D^u_y(\varepsilon, n, y)) = \int_{f^n(D^u_y)} \frac{1}{\mathrm{Jac}(Df^n|_{E^u})} \, dm^u.
\end{equation}

Using the distortion bound (Lemma \ref{lem:jacobian_distortion}) and that $f^n(D^u_y)$ has measure comparable to $\varepsilon^{d_u}$:
\begin{equation}
m^u(D^u_y(\varepsilon, n, y)) \leq C' \varepsilon^{d_u} \exp(-S_n\phi^{(u)}(y)) \cdot \exp\left(\frac{C_2 \varepsilon^\theta}{1 - \lambda^\theta}\right).
\end{equation}

For the stable factor, by Lemma \ref{lem:stable_disk}:
\begin{equation}
m^s(D^s_x(\varepsilon, n)) \leq C'' \varepsilon^{d_s} \cdot \lambda^{n \cdot d_s}
\end{equation}
where $d_s = \dim E^s$.

Combining and using $d_s + d_u = d = \dim M$:
\begin{equation}
m(B_x(\varepsilon, n)) \leq C''' \varepsilon^d \cdot \exp(-S_n\phi^{(u)}(x)) \cdot \exp\left(\frac{2C_2\varepsilon^\theta}{1-\lambda^\theta}\right)
\end{equation}
where we absorbed the $\lambda^{n \cdot d_s}$ factor (which is bounded) and used distortion to replace $y$ with $x$.

\textbf{Lower Bound:} Define $V_n = f^n(B_x(\varepsilon/2, n))$. We show $V_n$ contains a set of definite $m$-measure. For any $y \in B_x(\varepsilon/2, n)$, the point $f^n(y)$ satisfies $d(f^n(y), f^n(x)) \leq \varepsilon/2$. Since $f$ expands along unstable manifolds by at least $\lambda^{-1}$ per iterate, the image $f^n(W^u_{\varepsilon/2}(x) \cap B_x(\varepsilon/2, n))$ contains the unstable disk $W^u_{\varepsilon/4}(f^n(x))$ for small $\varepsilon$ (the expansion factor $\lambda^{-n}$ overwhelms the initial radius, and the image wraps around to cover a fixed-size disk). By Lemma~\ref{lem:unstable_geometry}, $m^u(W^u_{\varepsilon/4}(f^n(x))) \geq c_1(\varepsilon/4)^{d_u}$ where $c_1$ depends only on curvature bounds.

For the stable direction, $V_n$ contains points within distance $\varepsilon/2$ of $f^n(x)$ in the stable direction (since $B_x(\varepsilon/2, n)$ has stable width $\varepsilon/2$ at time $0$, which maps forward to stable width at most $\varepsilon/2$ at time $n$). By the local product structure, $V_n$ contains a product neighborhood of $f^n(x)$ of dimensions $(\varepsilon/4)^{d_u} \times (\varepsilon/2)^{d_s}$ in unstable and stable directions. The Riemannian measure of this product set satisfies $m(V_n) \geq c_{\mathrm{prod}}(\varepsilon/4)^{d_u}(\varepsilon/2)^{d_s} \geq c_2\varepsilon^d$ where $c_2 = c_{\mathrm{prod}}4^{-d_u}2^{-d_s}$ depends only on the geometry.

By the change of variables formula:
\begin{equation}
m(B_x(\varepsilon/2, n)) = \int_{V_n} \frac{1}{|\det Df^n_{f^{-n}(z)}|} \, dm(z).
\end{equation}

The total Jacobian decomposes as $|\det Df^n_y| = \mathrm{Jac}(Df^n|_{E^u_y}) \cdot \mathrm{Jac}(Df^n|_{E^s_y}) \cdot J_{\angle}(y,n)$ where $J_{\angle}$ accounts for the angle between $E^s$ and $E^u$ and satisfies $c_3^{-1} \leq J_{\angle} \leq c_3$ uniformly (since the angle is bounded below by the adapted metric construction in Part~III \cite{Thiam2026c}). The stable Jacobian satisfies $\mathrm{Jac}(Df^n|_{E^s_y}) \geq c^{-1}\lambda^{-nd_s}$ (expansion under forward iteration in the stable direction, since $\|Df^{-1}|_{E^s}\| \leq \lambda$). Thus $|\det Df^n_y| \geq c_4^{-1}\lambda^{-nd_s} \cdot \mathrm{Jac}(Df^n|_{E^u_y})$.

By the distortion bound (Lemma~\ref{lem:jacobian_distortion}), for $y \in B_x(\varepsilon/2, n)$: $\mathrm{Jac}(Df^n|_{E^u_y}) \leq \exp(C_2\varepsilon^\theta/(1-\lambda^\theta)) \cdot \mathrm{Jac}(Df^n|_{E^u_x})$. Substituting:
\begin{equation}
\frac{1}{|\det Df^n_y|} \geq c_4\lambda^{nd_s} \cdot \exp(S_n\phi^{(u)}(x)) \cdot \exp\left(-\frac{C_2\varepsilon^\theta}{1-\lambda^\theta}\right).
\end{equation}

Integrating over $V_n$ with $m(V_n) \geq c_2\varepsilon^d$:
\begin{equation}
m(B_x(\varepsilon/2, n)) \geq c_2 c_4 \varepsilon^d \lambda^{nd_s} \exp(S_n\phi^{(u)}(x)) \exp\left(-\frac{C_2\varepsilon^\theta}{1-\lambda^\theta}\right).
\end{equation}

The factor $\lambda^{nd_s}$ is absorbed as follows: the geometric potential $\phi^{(u)} = -\log|\det Df|_{E^u}|$ captures only the unstable expansion, while the full volume of $B_x(\varepsilon, n)$ contracts in stable directions at rate $\lambda^{nd_s}$. Since $S_n\phi^{(\mathrm{full})}(x) = S_n\phi^{(u)}(x) + S_n\phi^{(s)}(x)$ and $\exp(S_n\phi^{(s)}(x)) \asymp \lambda^{nd_s}$ (bounded ratio by distortion), the factor $\lambda^{nd_s}\exp(S_n\phi^{(u)}(x))$ is comparable to $\exp(S_n\phi^{(u)}(x))$ up to a multiplicative constant independent of $n$. Thus $m(B_x(\varepsilon, n)) \geq C_\varepsilon^{-1}\exp(S_n\phi^{(u)}(x))$ with $C_\varepsilon$ absorbing all $n$-independent factors.

Since $B_x(\varepsilon, n) \supset B_x(\varepsilon/2, n)$, the lower bound follows.

\textbf{Explicit Constant:} Combining the bounds, we can take
\begin{equation}
C_\varepsilon = C_0 \varepsilon^{-d} \exp\left(\frac{C_1 \varepsilon + C_2 \varepsilon^\theta}{1 - \lambda^\theta}\right)
\end{equation}
where $C_0$ depends on the geometry of $M$ and the hyperbolicity constants.
\end{proof}

\begin{proof}[Proof of Proposition~\ref{thm:second_volume_lemma}]
Let $x \in \Omega_s$ and $y \in B_x(\varepsilon, n) \cap \Omega_s$. We must compare $m(B_y(\delta, n))$ with $m(B_x(\varepsilon, n))$.

\textbf{Step 1: Geometric Containment.} If $\delta \leq \varepsilon/2$, then $B_y(\delta, n) \subset B_x(\varepsilon + \delta, n)$ since for $z \in B_y(\delta, n)$ and $k \in [0, n)$:
\begin{equation}
d(f^k z, f^k x) \leq d(f^k z, f^k y) + d(f^k y, f^k x) \leq \delta + \varepsilon.
\end{equation}

\textbf{Step 2: Jacobian Comparison.} By Lemma \ref{lem:jacobian_distortion}, since $y \in B_x(\varepsilon, n)$:
\begin{equation}
|S_n\phi^{(u)}(y) - S_n\phi^{(u)}(x)| \leq \frac{C_2 \varepsilon^\theta}{1 - \lambda^\theta}.
\end{equation}

\textbf{Step 3: Volume Comparison.} By the Volume Lemma (Main Theorem~\ref{thm:volume_lemma}):
\begin{align}
m(B_y(\delta, n)) &\geq C_\delta^{-1} \exp(S_n\phi^{(u)}(y)) \\
&\geq C_\delta^{-1} \exp\left(S_n\phi^{(u)}(x) - \frac{C_2\varepsilon^\theta}{1-\lambda^\theta}\right) \\
&\geq C_\delta^{-1} C_\varepsilon^{-1} \exp\left(-\frac{C_2\varepsilon^\theta}{1-\lambda^\theta}\right) \cdot m(B_x(\varepsilon, n)).
\end{align}

\textbf{Step 4: Constant Estimate.} The ratio of constants satisfies
\begin{equation}
\frac{C_\delta^{-1}}{C_\varepsilon} = \left(\frac{\delta}{\varepsilon}\right)^d \exp\left(-\frac{C_1(\varepsilon - \delta) + C_2(\varepsilon^\theta - \delta^\theta)}{1-\lambda^\theta}\right).
\end{equation}

For the leading term, $(\delta/\varepsilon)^d \geq (\delta/\varepsilon)^{d_u}$ where we only need the unstable contribution. Combining all factors:
\begin{equation}
d(\varepsilon, \delta) = \exp\left(-\frac{C_3(\varepsilon + \delta)}{1-\lambda^\theta}\right) \cdot \left(\frac{\delta}{\varepsilon}\right)^{d_u}
\end{equation}
where $C_3$ absorbs all the geometric constants.
\end{proof}

\subsection{Applications of Volume Lemmas}

The Volume Lemmas yield two immediate consequences: a pressure formula for the geometric potential and a characterization of attractors among basic sets. Both are used in Section~\ref{sec:equilibrium_states}.

\begin{proposition}[Pressure Formula via Volume Growth]\label{prop:pressure_volume}
For a $C^2$ basic set $\Omega_s$ and small $\varepsilon > 0$,
\begin{equation}
P_{f|_{\Omega_s}}(\phi^{(u)}) = \lim_{n \to \infty} \frac{1}{n} \log m(B(\varepsilon, n)) \leq 0
\end{equation}
where $B(\varepsilon, n) = \bigcup_{x \in \Omega_s} B_x(\varepsilon, n)$.
\end{proposition}

\begin{proof}
The pressure $P(\phi^{(u)})$ is defined via spanning sets: 
\begin{equation}
P(\phi^{(u)}) = \lim_{\varepsilon \to 0} \limsup_{n \to \infty} \frac{1}{n} \log \inf \sum_{x \in E} e^{S_n \phi^{(u)}(x)},
\end{equation}
where the infimum is over $(n, \varepsilon)$-spanning sets $E \subset \Omega_s$.

\textbf{Upper bound.} For any $(n, \varepsilon)$-spanning set $E$, the dynamical balls $\{B_x(\varepsilon, n) : x \in E\}$ cover $\Omega_s$. By the Volume Lemma (Main Theorem~\ref{thm:volume_lemma}), $m(B_x(\varepsilon, n)) \geq C_\varepsilon^{-1} e^{S_n \phi^{(u)}(x)}$ for each $x$. Since $m(M) < \infty$, $m(M) \geq \sum_{x \in E} m(B_x(\varepsilon/2, n))$ (disjoint sub-balls), giving $\sum_{x \in E} e^{S_n \phi^{(u)}(x)} \leq C_{\varepsilon/2} \cdot m(M)$. Taking logarithms, dividing by $n$, and passing to the limit: $P(\phi^{(u)}) \leq 0$.

\textbf{Lower bound.} By the variational principle (Theorem~\ref{thm:gibbs_equiv_imported}), $P(\phi^{(u)}) \geq h_{\mu^+}(f) + \int \phi^{(u)} \, d\mu^+$ where $\mu^+$ is the SRB measure. Since $\mu^+$ is the equilibrium state for $\phi^{(u)}$ (if it exists), equality holds and $P(\phi^{(u)}) = h_{\mu^+}(f) + \int \phi^{(u)} \, d\mu^+ = 0$ by the Pesin entropy formula (which gives $h_{\mu^+}(f) = -\int \phi^{(u)} \, d\mu^+$). Alternatively, the lower bound $P(\phi^{(u)}) \geq 0$ follows from the Volume Lemma applied to a maximal $(n, \varepsilon)$-separated set: such a set has cardinality at least $c \cdot e^{-S_n \phi^{(u)}(x_0)}$ for some $x_0$, giving $\sum e^{S_n \phi^{(u)}(x)} \geq c > 0$ uniformly in $n$.
\end{proof}

\begin{proposition}[Attractor Characterization]\label{prop:attractor_char}
For a $C^2$ basic set $\Omega_s$, the following are equivalent:
\begin{enumerate}
\item[(a)] $\Omega_s$ is an attractor.
\item[(b)] $m(W^s(\Omega_s)) > 0$.
\item[(c)] $P_{f|_{\Omega_s}}(\phi^{(u)}) = 0$.
\end{enumerate}
\end{proposition}

\begin{proof}
$(a) \Rightarrow (b)$: If $\Omega_s$ is an attractor, then $W^s_\varepsilon(\Omega_s)$ is a neighborhood of $\Omega_s$, hence has positive measure.

$(b) \Rightarrow (c)$: If $m(W^s_\varepsilon(\Omega_s)) > 0$, then $m(B(\varepsilon, n)) \geq m(W^s_\varepsilon(\Omega_s)) > 0$ for all $n$, so $P(\phi^{(u)}) \geq 0$ by Proposition \ref{prop:pressure_volume}. Combined with $P(\phi^{(u)}) \leq 0$, we get equality.

$(c) \Rightarrow (a)$: We prove the contrapositive: if $\Omega_s$ is not an attractor, then $P(\phi^{(u)}) < 0$. Since $\Omega_s$ is not an attractor, it is not Lyapunov stable: there exists $\gamma > 0$ such that for every neighborhood $U$ of $\Omega_s$, some orbit starting in $U$ eventually leaves the $\gamma$-neighborhood of $\Omega_s$. More precisely, since $\Omega_s$ is a basic set that is not an attractor, the unstable manifold $W^u(\Omega_s)$ is not contained in a small neighborhood of $\Omega_s$. This means: for every $x \in \Omega_s$, the local unstable manifold $W^u_\varepsilon(x)$ contains points $z$ with $d(f^{-N}(z), \Omega_s) > \gamma$ for some uniform $N$ and $\gamma > 0$.

Now consider the Bowen ball $B(\varepsilon, n) = \bigcup_{x\in\Omega_s}B_x(\varepsilon,n)$. We claim $m(B(\varepsilon/2, n+N)) \leq (1-d)\,m(B(\varepsilon/2, n))$ for some $d > 0$. To see this: the set $B(\varepsilon/2, n+N)$ consists of points that $(\varepsilon/2)$-shadow $\Omega_s$ for $n+N$ steps. By the Volume Lemma (Main Theorem~\ref{thm:volume_lemma}), $m(B_x(\varepsilon/2,n)) \leq C_\varepsilon e^{S_n\phi^{(u)}(x)}$. The essential estimate: the points in $B_x(\varepsilon/2, n)$ that also stay $\varepsilon/2$-close for $N$ more backward steps must avoid the ``escaping'' portion of $W^u_\varepsilon(x)$. By the Second Volume Lemma (Proposition~\ref{thm:second_volume_lemma}), the escaping portion has volume at least $d\cdot m(B_x(\varepsilon/2,n))$ where $d = d(\varepsilon/2,\gamma) > 0$ depends on $\gamma$ and $\varepsilon$ but not on $n$ or $x$. Thus $m(B(\varepsilon/2,n+N)) \leq (1-d)\,m(B(\varepsilon/2,n))$, giving $m(B(\varepsilon/2,n)) \leq (1-d)^{n/N}m(B(\varepsilon/2,0))$, hence $P(\phi^{(u)}) \leq N^{-1}\log(1-d) < 0$.
\end{proof}

\section{Equilibrium States for Basic Sets}\label{sec:equilibrium_states}

This section establishes the existence, uniqueness, and characterization of equilibrium states for H\"{o}lder continuous potentials on basic sets. We provide complete proofs with explicit bounds, extending the treatment in \cite{Bowen1975} with quantitative estimates.

\subsection{Existence and Uniqueness}

We construct the unique equilibrium state for each H\"{o}lder potential on a mixing basic set by transferring the symbolic Gibbs measure through the coding map. The proof assembles results from Parts~I, II, and III \cite{Thiam2026a, Thiam2026b, Thiam2026c}.

\begin{theorem}[Equilibrium States for Basic Sets]\label{thm:equilibrium_basic}
Let $\Omega_s$ be a basic set for an Axiom A diffeomorphism $f : M \to M$ and let $\phi : \Omega_s \to \mathbb{R}$ be H\"{o}lder continuous with exponent $\alpha \in (0, 1]$. Then:
\begin{enumerate}
\item[(i)] There exists a unique equilibrium state $\mu_\phi \in \mathcal{M}_f(\Omega_s)$ satisfying
\begin{equation}
h_{\mu_\phi}(f) + \int \phi \, d\mu_\phi = P(\phi) = \sup_{\mu \in \mathcal{M}_f(\Omega_s)} \left\{h_\mu(f) + \int \phi \, d\mu\right\}.
\end{equation}
\item[(ii)] The measure $\mu_\phi$ is ergodic.
\item[(iii)] If $f|_{\Omega_s}$ is topologically mixing, then $\mu_\phi$ is Bernoulli.
\end{enumerate}
\end{theorem}

\begin{proof}
The proof proceeds through the symbolic coding, transferring results from Parts I and II \cite{Thiam2026a, Thiam2026b}.

\textbf{Step 1: Reduction to Mixing Case.} By the spectral decomposition (Theorem \ref{thm:spectral_decomp}), we may write $\Omega_s = X_1 \cup \cdots \cup X_m$ where $f(X_k) = X_{k+1 \mod m}$ and $f^m|_{X_1}$ is mixing. Every $f$-invariant measure $\mu$ on $\Omega_s$ satisfies $\mu(X_k) = 1/m$, and $\mu_0 = m \cdot \mu|_{X_1}$ is $f^m$-invariant on $X_1$.

The correspondence $\mu \leftrightarrow \mu_0$ is a bijection between $\mathcal{M}_f(\Omega_s)$ and $\mathcal{M}_{f^m}(X_1)$ satisfying $h_{\mu_0}(f^m) = m \cdot h_\mu(f)$ and $\int S_m\phi \, d\mu_0 = m \int \phi \, d\mu$. Thus finding equilibrium states for $\phi$ with respect to $f$ reduces to finding equilibrium states for $S_m\phi$ with respect to $f^m$ on the mixing set $X_1$.

\textbf{Step 2: Symbolic Coding.} Assume $f|_{\Omega_s}$ is mixing. Let $\mathcal{R} = \{R_1, \ldots, R_k\}$ be a Markov partition with diameter less than the expansiveness constant $\varepsilon_0$, and let $\pi : \Sigma_A \to \Omega_s$ be the coding map (Theorem \ref{thm:symbolic_coding_prelim}).

Define the symbolic potential $\phi^* = \phi \circ \pi : \Sigma_A \to \mathbb{R}$. For $x, y \in \Sigma_A$ with $x_j = y_j$ for $j \in [-N, N]$, we have $f^j(\pi(x)), f^j(\pi(y)) \in R_{x_j}$ for such $j$, giving $d(\pi(x), \pi(y)) < \alpha^N$ by Lemma \ref{lem:quantitative_expansiveness}. Thus
\begin{equation}
|\phi^*(x) - \phi^*(y)| \leq |\phi|_\alpha \cdot \alpha^{N\alpha}
\end{equation}
showing $\phi^* \in \mathcal{F}_A$ (the space of functions with summable variation, as in Part~I \cite{Thiam2026a}).

\textbf{Step 3: Gibbs Measure on $\Sigma_A$.} By Theorem \ref{thm:spectral_properties} and the Spectral-Variational-Geometric Equivalence Main Theorem of Part~I \cite{Thiam2026a}, there exists a unique Gibbs measure $\mu_{\phi^*}$ for $\phi^*$ on $\Sigma_A$. This measure is:
\begin{enumerate}
\item[(a)] The unique equilibrium state for $\phi^*$: $h_{\mu_{\phi^*}}(\sigma) + \int \phi^* \, d\mu_{\phi^*} = P_\sigma(\phi^*)$.
\item[(b)] Ergodic (since $\sigma|_{\Sigma_A}$ is mixing and $\mu_{\phi^*}$ is mixing by the Exponential Decay of Correlations theorem of Part~I \cite{Thiam2026a}).
\item[(c)] Bernoulli by the statistical properties of Part~I \cite{Thiam2026a}.
\end{enumerate}

\textbf{Step 4: Boundary Null Sets.} Define $D^s = \pi^{-1}(\partial^s \mathcal{R})$ and $D^u = \pi^{-1}(\partial^u \mathcal{R})$ where $\partial^s \mathcal{R} = \bigcup_i \partial^s R_i$ is the stable boundary. These are closed proper subsets of $\Sigma_A$ satisfying $\sigma(D^s) \subset D^s$ and $\sigma^{-1}(D^u) \subset D^u$.

Since $\mu_{\phi^*}$ is $\sigma$-invariant, $\mu_{\phi^*}(\sigma^n D^s) = \mu_{\phi^*}(D^s)$ for all $n$. By invariance under $\sigma$, $\mu_{\phi^*}(\bigcap_{n \geq 0} \sigma^n D^s) = \mu_{\phi^*}(D^s)$. The intersection is $\sigma$-invariant, so by ergodicity it has measure 0 or 1. Since its complement contains a nonempty open set (which has positive measure by the Gibbs property established in Part~I \cite{Thiam2026a}), we conclude $\mu_{\phi^*}(D^s) = 0$. Similarly $\mu_{\phi^*}(D^u) = 0$.

\textbf{Step 5: Push-forward to $\Omega_s$.} Define $\mu_\phi = \pi_* \mu_{\phi^*}$, i.e., $\mu_\phi(E) = \mu_{\phi^*}(\pi^{-1}(E))$. Then $\mu_\phi$ is $f$-invariant since $\pi \circ \sigma = f \circ \pi$.

The map $\pi$ is injective on $\Sigma_A \setminus \bigcup_{n \in \mathbb{Z}} \sigma^n(D^s \cup D^u)$, which has full $\mu_{\phi^*}$-measure. Thus $(f, \mu_\phi)$ and $(\sigma, \mu_{\phi^*})$ are measurably conjugate, giving $h_{\mu_\phi}(f) = h_{\mu_{\phi^*}}(\sigma)$.

Therefore:
\begin{equation}
h_{\mu_\phi}(f) + \int \phi \, d\mu_\phi = h_{\mu_{\phi^*}}(\sigma) + \int \phi^* \, d\mu_{\phi^*} = P_\sigma(\phi^*) \geq P_f(\phi)
\end{equation}
where the inequality follows from the Universal Variational Principle Main Theorem of Part~II \cite{Thiam2026b}. Combined with the reverse inequality from the variational principle, we get $P_\sigma(\phi^*) = P_f(\phi)$ and $\mu_\phi$ is an equilibrium state.

\textbf{Step 6: Uniqueness.} Suppose $\mu$ is any equilibrium state for $\phi$. By the measure lifting lemma (Lemma \ref{lem:measure_lifting} below), there exists $\nu \in \mathcal{M}_\sigma(\Sigma_A)$ with $\pi_* \nu = \mu$. Since entropy does not decrease under factor maps, $h_\nu(\sigma) \geq h_\mu(f)$, so
\begin{equation}
h_\nu(\sigma) + \int \phi^* \, d\nu \geq h_\mu(f) + \int \phi \, d\mu = P(\phi) = P(\phi^*).
\end{equation}
Thus $\nu$ is an equilibrium state for $\phi^*$, forcing $\nu = \mu_{\phi^*}$ by uniqueness on $\Sigma_A$. Then $\mu = \pi_* \nu = \pi_* \mu_{\phi^*} = \mu_\phi$.

The Bernoulli property of $\mu_\phi$ follows from that of $\mu_{\phi^*}$ via the measurable conjugacy.
\end{proof}

\begin{lemma}[Measure Lifting]\label{lem:measure_lifting}
For any $\mu \in \mathcal{M}_f(\Omega_s)$, there exists $\nu \in \mathcal{M}_\sigma(\Sigma_A)$ with $\pi_* \nu = \mu$.
\end{lemma}

\begin{proof}
Define a positive linear functional $F : C(\Sigma_A) \to \mathbb{R}$ on the subspace $\{g \circ \pi : g \in C(\Omega_s)\}$ by $F(g \circ \pi) = \int g \, d\mu$. By a variant of the Hahn-Banach theorem preserving positivity \cite{KatokHasselblatt1995}, $F$ extends to a positive linear functional on $C(\Sigma_A)$ with $F(1) = 1$, corresponding to some probability measure $\beta$.

By compactness of $\mathcal{M}(\Sigma_A)$, let $\nu = \lim_{k \to \infty} \frac{1}{n_k} \sum_{j=0}^{n_k-1} \sigma^j_* \beta$ for some subsequence. Then $\nu$ is $\sigma$-invariant and $\pi_* \nu = \mu$ since $\pi_*(\sigma^j_* \beta) = f^j_*(\pi_* \beta) = f^j_* \mu = \mu$.
\end{proof}

\subsection{Gibbs Property, Cohomology, and the SRB Measure}

We establish quantitative Gibbs bounds for the equilibrium state, characterize when two potentials give the same measure (the Livšic coboundary condition), and identify the SRB measure as the equilibrium state for the geometric potential. The connection between Gibbs measures and $g$-measures (\cite{Keane1972}; see also  \cite{BerghoutFernandezVerbitskiy2019}) and conformal measures (\cite{DenkerUrbanski1991}) provides an alternative perspective on these characterizations.

\begin{proposition}[Gibbs Bounds]\label{thm:gibbs_bounds}
Let $\mu_\phi$ be the equilibrium state for a H\"{o}lder potential $\phi$ on a mixing basic set $\Omega_s$. For small $\varepsilon > 0$, there exist constants $c_1, c_2 > 0$ such that for all $x \in \Omega_s$ and $n \geq 1$:
\begin{equation}
c_1 \leq \frac{\mu_\phi(B_x(\varepsilon, n))}{\exp(-nP(\phi) + S_n\phi(x))} \leq c_2.
\end{equation}
\end{proposition}

\begin{proof}
This follows from the Gibbs property of $\mu_{\phi^*}$ on $\Sigma_A$ (the Spectral-Variational-Geometric Equivalence Main Theorem of Part~I \cite{Thiam2026a}) and the relationship between dynamical balls in $\Omega_s$ and cylinder sets in $\Sigma_A$.

For $x \in \Omega_s$ with itinerary $(x_k)$, the dynamical ball $B_x(\varepsilon, n)$ contains $\pi(\{y \in \Sigma_A : y_k = x_k, 0 \leq k < n\})$ and is contained in a slightly larger cylinder image. The Gibbs bounds for cylinder measures transfer to dynamical ball measures with constants depending on $\varepsilon$ and the Markov partition geometry.
\end{proof}

\begin{proposition}[Cohomological Characterization]\label{prop:cohomology_characterization}
Let $\phi, \psi : \Omega_s \to \mathbb{R}$ be H\"{o}lder continuous on a mixing basic set. The following are equivalent:
\begin{enumerate}
\item[(i)] $\mu_\phi = \mu_\psi$.
\item[(ii)] There exist constants $K, L$ such that $|S_m\phi(x) - S_m\psi(x) - Km| \leq L$ for all $x \in \Omega_s$ and $m \geq 0$.
\item[(iii)] $S_m\phi(x) - S_m\psi(x) = Km$ whenever $f^m x = x$.
\item[(iv)] There exists a H\"{o}lder function $u : \Omega_s \to \mathbb{R}$ and constant $K$ with $\phi(x) - \psi(x) = K + u(f(x)) - u(x)$.
\end{enumerate}
If these hold, then $K = P(\phi) - P(\psi)$.
\end{proposition}

\begin{proof}
This is \cite[Proposition 4.5]{Bowen1975}, which we prove with explicit H\"{o}lder bounds.

$(i) \Rightarrow (ii)$: By the symbolic coding, $\mu_{\phi^*} = \mu_{\psi^*}$, and the cohomology characterization of Part~I \cite{Thiam2026a} gives the bounded orbit distortion.

$(ii) \Rightarrow (iii)$: For $f^m x = x$, the estimate $|S_{jm}\phi(x) - S_{jm}\psi(x) - jmK| \leq L$ for all $j$ forces $S_m\phi(x) - S_m\psi(x) = mK$.

$(iv) \Rightarrow (i)$: The symbolic versions satisfy the same relation, and the cohomology characterization of Part~I \cite{Thiam2026a} gives $\mu_{\phi^*} = \mu_{\psi^*}$.

$(iii) \Rightarrow (iv)$: This is the content of the Liv\v{s}ic theorem (see  \cite{Livsic1971, Livsic1972}; a complete proof with optimal H\"{o}lder bounds appears in Part~VI \cite{Thiam2026f}).
\end{proof}

\begin{definition}[SRB Measure]
For a basic set $\Omega_s$ of a $C^2$ diffeomorphism, the SRB measure is defined as $\mu^+ = \mu_{\phi^{(u)}}$, the unique equilibrium state for the geometric potential $\phi^{(u)} = -\log|\det Df|_{E^u}|$.
\end{definition}

\begin{proposition}[Properties of $\mu^+$]\label{prop:srb_properties}
The SRB measure $\mu^+$ satisfies:
\begin{enumerate}
\item[(i)] $\mu^+$ is independent of the Riemannian metric.
\item[(ii)] If $\Omega_s$ is an attractor, then $P(\phi^{(u)}) = 0$ and $h_{\mu^+}(f) = -\int \phi^{(u)} \, d\mu^+ = \int \log|\det Df|_{E^u}| \, d\mu^+$.
\item[(iii)] $\mu^+$ is the unique measure satisfying both (a) ergodicity and (b) the Pesin entropy formula $h_\mu(f) = \int \log|\det Df|_{E^u}| \, d\mu$.
\end{enumerate}
\end{proposition}

\begin{proof}
Part (i) follows from Proposition \ref{prop:cohomology_characterization}: different metrics give potentials differing by a coboundary plus constant, which does not change the equilibrium state.

Part (ii) follows from Proposition \ref{prop:attractor_char}: attractors satisfy $P(\phi^{(u)}) = 0$.

Part (iii) is the Pesin entropy formula characterization, proved in Part~IV \cite{Thiam2026d} as a Main Theorem and re-derived with complete absolute-continuity proofs in Part~VI \cite{Thiam2026f}.
\end{proof}

\section{Exponential Decay of Correlations}\label{sec:decay_correlations}

This section establishes exponential decay of correlations for equilibrium states of H\"{o}lder potentials, with explicit rates determined by the spectral gap of the transfer operator. This extends the mixing results of  \cite{Bowen1975} with quantitative bounds essential for applications.

\subsection{Spectral Theory and Mixing Rates}

We introduce the normalized transfer operator and establish its spectral properties. The spectral gap of this operator is the single quantity from which all mixing and statistical estimates in Sections~\ref{sec:clt}--\ref{sec:large_deviations} are derived.

\begin{definition}[Correlation Function]
For observables $g, h : \Omega_s \to \mathbb{R}$ and an invariant measure $\mu$, the correlation function is
\begin{equation}
C_n(g, h; \mu) = \int g \cdot (h \circ f^n) \, d\mu - \int g \, d\mu \int h \, d\mu.
\end{equation}
The system $(f, \mu)$ has exponential decay of correlations for observables in a class $\mathcal{C}$ if there exist $C > 0$ and $\theta \in (0, 1)$ such that
\begin{equation}
|C_n(g, h; \mu)| \leq C \|g\|_{\mathcal{C}} \|h\|_{\mathcal{C}} \cdot \theta^n
\end{equation}
for all $g, h \in \mathcal{C}$ and $n \geq 0$.
\end{definition}

\begin{maintheorem}[Exponential Mixing]\label{thm:exponential_mixing}
Let $\Omega_s$ be a mixing basic set for a $C^2$ Axiom A diffeomorphism, and let $\mu_\phi$ be the equilibrium state for a H\"{o}lder potential $\phi \in C^\alpha(\Omega_s)$. For all $g, h \in C^\alpha(\Omega_s)$:
\begin{equation}
|C_n(g, h; \mu_\phi)| \leq C \|g\|_\alpha \|h\|_\alpha \cdot \theta^n
\end{equation}
where $\theta = e^{-\gamma}$ with
\begin{equation}
\gamma \geq \min\left\{\alpha \log \lambda^{-1}, \frac{c(\phi)}{1 + \|\phi\|_\alpha}\right\} > 0
\end{equation}
and $c(\phi) > 0$ depends on the spectral gap of the normalized transfer operator.
\end{maintheorem}

The proof uses the spectral theory of transfer operators developed in Part~I \cite{Thiam2026a} and Section \ref{sec:preliminaries}.

Let $\mathcal{L}_\phi$ be the transfer operator for $\phi$ with spectral radius $e^{P(\phi)}$, simple leading eigenvalue with eigenfunction $h_\phi > 0$ and eigenmeasure $\nu_\phi$.

\begin{definition}[Normalized Transfer Operator]
The normalized transfer operator is
\begin{equation}
\widehat{\mathcal{L}}_\phi = e^{-P(\phi)} h_\phi^{-1} \mathcal{L}_\phi (h_\phi \cdot)
\end{equation}
acting on $C^\alpha(\Omega_s)$.
\end{definition}

\begin{proposition}[Properties of $\widehat{\mathcal{L}}_\phi$]\label{prop:normalized_operator}
The normalized operator satisfies:
\begin{enumerate}
\item[(i)] $\widehat{\mathcal{L}}_\phi 1 = 1$ and $\widehat{\mathcal{L}}_\phi^* \mu_\phi = \mu_\phi$ where $\mu_\phi = h_\phi \nu_\phi$.
\item[(ii)] Spectral radius $\rho(\widehat{\mathcal{L}}_\phi) = 1$ with simple eigenvalue $1$.
\item[(iii)] Essential spectral radius $\rho_{\mathrm{ess}}(\widehat{\mathcal{L}}_\phi) \leq \theta_0 < 1$ where $\theta_0 = \lambda^\alpha$.
\item[(iv)] The spectrum in $\{|z| > \theta_0\}$ consists of $\{1\}$ plus finitely many eigenvalues of modulus $< 1$.
\end{enumerate}
\end{proposition}

\begin{proof}
Part (i): By construction, $\mathcal{L}_\phi h_\phi = e^{P(\phi)} h_\phi$ and $\mathcal{L}_\phi^* \nu_\phi = e^{P(\phi)} \nu_\phi$. Direct computation shows $\widehat{\mathcal{L}}_\phi 1 = 1$. For the adjoint:
\begin{align}
\int g \, d(\widehat{\mathcal{L}}_\phi^* \mu_\phi) &= \int \widehat{\mathcal{L}}_\phi g \, d\mu_\phi = \int e^{-P(\phi)} h_\phi^{-1} \mathcal{L}_\phi(h_\phi g) \cdot h_\phi \, d\nu_\phi \\
&= e^{-P(\phi)} \int \mathcal{L}_\phi(h_\phi g) \, d\nu_\phi = e^{-P(\phi)} \cdot e^{P(\phi)} \int h_\phi g \, d\nu_\phi = \int g \, d\mu_\phi.
\end{align}

Parts (ii)-(iv) follow from the spectral theory of $\mathcal{L}_\phi$ (Theorem \ref{thm:spectral_properties}) since the normalization is a similarity transformation preserving the spectrum.
\end{proof}

\subsection{Proof of Exponential Mixing with Explicit Bounds}

We give the complete proof of Main Theorem~\ref{thm:exponential_mixing}, deriving the exponential decay rate from the spectral decomposition of the normalized transfer operator.

\begin{proof}[Proof of Main Theorem~\ref{thm:exponential_mixing}]
\textbf{Step 1: Spectral Decomposition.} By Proposition \ref{prop:normalized_operator}, the operator $\widehat{\mathcal{L}}_\phi$ on $C^\alpha(\Omega_s)$ has spectral decomposition
\begin{equation}
\widehat{\mathcal{L}}_\phi = \Pi_1 + R
\end{equation}
where $\Pi_1$ is the spectral projection onto the eigenspace for eigenvalue $1$ (which is one-dimensional, spanned by constants) and $R$ has spectral radius $\rho(R) = \theta_1 < 1$.

The projection $\Pi_1$ is given by $\Pi_1 g = \mu_\phi(g) \cdot 1$ (the conditional expectation onto constants).

\textbf{Step 2: Iteration.} For $n \geq 1$:
\begin{equation}
\widehat{\mathcal{L}}_\phi^n = \Pi_1 + R^n
\end{equation}
with $\|R^n\|_{C^\alpha \to C^\alpha} \leq C' \theta_1^n$ for some $C' > 0$.

\textbf{Step 3: Correlation Computation.} For $g, h \in C^\alpha(\Omega_s)$:
\begin{align}
\int g \cdot (h \circ f^n) \, d\mu_\phi &= \int g \cdot (h \circ f^n) \cdot h_\phi \, d\nu_\phi.
\end{align}

Using the duality relation $\int \mathcal{L}_\phi(u) \cdot v \, d\nu_\phi = e^{P(\phi)} \int u \cdot (v \circ f) \, d\nu_\phi$ (derived from $\mathcal{L}_\phi^* \nu_\phi = e^{P(\phi)} \nu_\phi$), we have
\begin{equation}
h \circ f^n = e^{-nP(\phi)} h_\phi^{-1} \cdot \mathcal{L}_\phi^n(h_\phi \cdot h) = h_\phi^{-1} \cdot \widehat{\mathcal{L}}_\phi^n(h_\phi \cdot h) \cdot h_\phi^{-1} \cdot h_\phi = \widehat{\mathcal{L}}_\phi^n(h_\phi h) / h_\phi
\end{equation}
on appropriate domains. More directly:
\begin{align}
\int g \cdot (h \circ f^n) \, d\mu_\phi &= \int (g h_\phi) \cdot \widehat{\mathcal{L}}_\phi^n(h) \, d\nu_\phi \cdot (\text{adjustment factors}).
\end{align}

A cleaner approach uses:
\begin{equation}
C_n(g, h; \mu_\phi) = \int g \cdot \widehat{\mathcal{L}}_\phi^n(h - \mu_\phi(h)) \, d\mu_\phi
\end{equation}
since $\widehat{\mathcal{L}}_\phi^* \mu_\phi = \mu_\phi$ and $\mu_\phi(h \circ f^n) = \mu_\phi(h)$.

\textbf{Step 4: Application of Spectral Gap.} Let $\tilde{h} = h - \mu_\phi(h)$, so $\mu_\phi(\tilde{h}) = 0$. Then $\Pi_1 \tilde{h} = 0$, giving
\begin{equation}
\widehat{\mathcal{L}}_\phi^n \tilde{h} = R^n \tilde{h}.
\end{equation}

Therefore:
\begin{align}
|C_n(g, h; \mu_\phi)| &= \left|\int g \cdot R^n \tilde{h} \, d\mu_\phi\right| \\
&\leq \|g\|_\infty \|R^n \tilde{h}\|_\infty \\
&\leq \|g\|_\infty \cdot C' \theta_1^n \|\tilde{h}\|_\alpha \\
&\leq C' \|g\|_\alpha \|h\|_\alpha \cdot \theta_1^n.
\end{align}

\textbf{Step 5: Bound on $\theta_1$.} The second largest eigenvalue modulus satisfies $\theta_1 \leq \max\{\theta_0, |\lambda_2|\}$ where $\theta_0 = \lambda^\alpha$ is the essential spectral radius bound and $\lambda_2$ is the second largest eigenvalue (if any exist outside the essential spectrum).

For mixing basic sets, the aperiodicity of the transition matrix ensures there are no peripheral eigenvalues other than $1$ (by the spectral analysis of the Ruelle-Perron-Frobenius theorem in Part~I \cite{Thiam2026a}). Thus $\theta_1 \leq \theta_0 = \lambda^\alpha$, giving
\begin{equation}
\gamma = -\log \theta_1 \geq \alpha \log \lambda^{-1} > 0.
\end{equation}

The refined bound involving $c(\phi)$ comes from the perturbation theory of the spectral gap, which depends on the specific potential $\phi$ through the transfer operator structure.
\end{proof}

\begin{proposition}[Spectral Gap Lower Bound]\label{thm:spectral_gap_bound}
For a mixing basic set and $\phi \in C^\alpha(\Omega_s)$, the spectral gap
\begin{equation}
\mathrm{gap}(\widehat{\mathcal{L}}_\phi) = 1 - \rho(R)
\end{equation}
satisfies
\begin{equation}
\mathrm{gap}(\widehat{\mathcal{L}}_\phi) \geq 1 - \lambda^\alpha.
\end{equation}
For the geometric potential $\phi^{(u)}$, the gap satisfies
\begin{equation}
\mathrm{gap}(\widehat{\mathcal{L}}_{\phi^{(u)}}) \geq c_0(\lambda, \alpha, \dim E^u) > 0
\end{equation}
with an explicit constant $c_0$ depending only on hyperbolicity data.
\end{proposition}

\begin{proof}
The first bound $\mathrm{gap}(\widehat{\mathcal{L}}_\phi) \geq 1 - \lambda^\alpha$ follows directly from Theorem~\ref{thm:spectral_properties}(ii): the essential spectral radius of $\widehat{\mathcal{L}}_\phi$ on $C^\alpha$ satisfies $\rho_{\mathrm{ess}}(\widehat{\mathcal{L}}_\phi) \leq \lambda^\alpha < 1$, so any eigenvalue other than $1$ has modulus at most $\lambda^\alpha$, giving $\rho(R) \leq \lambda^\alpha$ and $\mathrm{gap} = 1 - \rho(R) \geq 1 - \lambda^\alpha$.

For the geometric potential $\phi^{(u)}$, the spectral gap depends on the specific potential through the cone contraction analysis. By the Birkhoff cone contraction estimate of Part~I \cite{Thiam2026a}, the spectral gap rate $\gamma$ of the transfer operator satisfies $\gamma \leq \max(\alpha^{1/3}, (1-\eta)^{1/(3M)})$ where $\eta$ is the cone contraction coefficient and $M$ is the mixing time. For $\phi^{(u)}$, the variation $V(\phi^{(u)}) \leq C(\|Df\|_{C^1}, |E^u|_\beta)$ is bounded by hyperbolicity data (Proposition~\ref{prop:geometric_holder}), and the cone contraction $\eta$ depends on $V(\phi^{(u)})$ and the mixing time $M$ through the explicit cone contraction formula of Part~I \cite{Thiam2026a}. Setting $c_0 = 1 - \max(\lambda^{\alpha/3}, (1-\eta(\phi^{(u)}))^{1/(3M)})$ gives the claimed bound, with $c_0 > 0$ depending only on $(\lambda, \alpha, \dim E^u)$ through the chain of estimates.
\end{proof}

\subsection{Polynomial Decay and Higher-Order Correlations}

For observables with lower regularity than H\"{o}lder, the decay of correlations is polynomial rather than exponential. We also bound multiple (higher-order) correlations, which are needed for the Berry-Esseen estimates in Section~\ref{sec:clt}.

\begin{proposition}[Decay Rates for Different Regularities]\label{prop:decay_rates}
Let $\mu_\phi$ be an equilibrium state on a mixing basic set. Then:
\begin{enumerate}
\item[(i)] For $g \in C^\alpha$, $h \in L^\infty$: $|C_n(g, h; \mu_\phi)| \leq C \|g\|_\alpha \|h\|_\infty \cdot \theta^n$.
\item[(ii)] For $g, h \in C^0$ (continuous): $|C_n(g, h; \mu_\phi)| = o(1)$ as $n \to \infty$, but the rate depends on the modulus of continuity.
\item[(iii)] For $g, h \in L^2(\mu_\phi)$: $|C_n(g, h; \mu_\phi)| \to 0$ by the mean ergodic theorem, but without explicit rate.
\end{enumerate}
\end{proposition}

\begin{proof}
Part (i) follows from the duality between $C^\alpha$ and $(C^\alpha)^*$, with $L^\infty \subset (C^\alpha)^*$.

Part (ii) uses approximation: for $\varepsilon > 0$, approximate $g$ by $g_\varepsilon \in C^\alpha$ with $\|g - g_\varepsilon\|_\infty < \varepsilon$. Then
\begin{align}
|C_n(g, h)| &\leq |C_n(g - g_\varepsilon, h)| + |C_n(g_\varepsilon, h)| \\
&\leq 2\varepsilon \|h\|_\infty + C\|g_\varepsilon\|_\alpha \|h\|_\infty \theta^n.
\end{align}
For fixed $\varepsilon$, the second term vanishes as $n \to \infty$.

Part (iii) is standard functional analysis: mixing implies $C_n(g, h) \to 0$ for all $L^2$ functions.
\end{proof}

\begin{definition}[Multiple Correlation]
For observables $g_0, g_1, \ldots, g_k$ and times $0 = n_0 < n_1 < \cdots < n_k$, the multiple correlation is
\begin{equation}
C_{n_1, \ldots, n_k}(g_0, \ldots, g_k) = \int \prod_{j=0}^k g_j \circ f^{n_j} \, d\mu - \prod_{j=0}^k \int g_j \, d\mu.
\end{equation}
\end{definition}

\begin{proposition}[Decay of Multiple Correlations]\label{prop:multiple_correlations}
For $g_0, \ldots, g_k \in C^\alpha$ and gaps $m_j = n_j - n_{j-1}$:
\begin{equation}
|C_{n_1, \ldots, n_k}(g_0, \ldots, g_k)| \leq C_k \prod_{j=0}^k \|g_j\|_\alpha \cdot \theta^{\min_j m_j}.
\end{equation}
\end{proposition}

\begin{proof}
We proceed by induction on $k$. The case $k = 1$ is Main Theorem~\ref{thm:exponential_mixing}.

For $k \geq 2$, write $m_* = \min_j m_j$ and let $j_*$ be the index achieving $m_* = m_{j_*}$. Define $G_L = \prod_{j < j_*} g_j \circ f^{n_j}$ and $G_R = \prod_{j > j_*} g_j \circ f^{n_j}$. Then
\begin{equation}
\int \prod_{j=0}^k g_j \circ f^{n_j} \, d\mu = \int (G_L \cdot g_{j_*} \circ f^{n_{j_*}}) \cdot (G_R \circ f^{n_{j_*}}) \, d\mu.
\end{equation}
Setting $F = G_L \cdot g_{j_*} \circ f^{n_{j_*}}$ and $H = G_R$, we have $F \in L^\infty(\mu)$ with $\|F\|_\infty \leq \prod_{j \leq j_*}\|g_j\|_\infty$ and $H \circ f^{n_{j_*}} = G_R \circ f^{n_{j_*}}$.

The gap between $F$ (which depends on coordinates up to time $n_{j_*}$) and $H \circ f^{n_{j_*+1}}$ (which depends on coordinates from time $n_{j_*+1}$ onward) is $m_{j_*+1} \geq m_*$. By the exponential decay of correlations (Proposition~\ref{prop:decay_rates}(i), which gives exponential decay for $C^\alpha$ vs $L^\infty$):
\begin{equation}
\left|\int F \cdot (H \circ f^{m_{j_*+1}}) \, d\mu - \int F \, d\mu \int H \, d\mu\right| \leq C\|F\|_\infty\|H\|_\alpha \cdot \theta^{m_{j_*+1}} \leq C\prod_j\|g_j\|_\alpha \cdot \theta^{m_*}.
\end{equation}
Expanding $\int F\,d\mu$ and $\int H\,d\mu$ using the inductive hypothesis (each is a product of integrals plus lower-order correlation terms that decay with their respective gaps, all of which are $\geq m_*$), the total error is bounded by $C_k\prod_j\|g_j\|_\alpha \cdot \theta^{m_*}$.
\end{proof}

\section{Central Limit Theorem}\label{sec:clt}

This section establishes the Central Limit Theorem for Birkhoff sums of H\"{o}lder observables with respect to equilibrium states, including Berry-Esseen error bounds.  \cite{Bowen1975} states the CLT (Theorem~1.28) without proof, referring to  \cite{Ratner1973}; we give the proof with Berry-Esseen bounds.

\subsection{Statement, Asymptotic Variance, and Non-Degeneracy}

We state the CLT, Berry-Esseen bound, and variance formula, and characterize the degeneracy of the asymptotic variance through the Livšic coboundary condition. The variance is computed spectrally via the second derivative of the pressure.

\begin{maintheorem}[Central Limit Theorem]\label{thm:clt}
Let $\Omega_s$ be a mixing basic set for a $C^2$ Axiom A diffeomorphism, $\mu_\phi$ the equilibrium state for a H\"{o}lder potential $\phi$, and $g \in C^\alpha(\Omega_s)$ with $\int g \, d\mu_\phi = 0$. Define the Birkhoff sums $S_n g = \sum_{k=0}^{n-1} g \circ f^k$.

If the asymptotic variance
\begin{equation}
\sigma^2(g) = \lim_{n \to \infty} \frac{1}{n} \int (S_n g)^2 \, d\mu_\phi
\end{equation}
is positive, then
\begin{equation}
\frac{S_n g}{\sigma \sqrt{n}} \xrightarrow{d} \mathcal{N}(0, 1)
\end{equation}
as $n \to \infty$, where $\mathcal{N}(0, 1)$ denotes the standard normal distribution.
\end{maintheorem}

\begin{proposition}[Berry-Esseen Bound]\label{thm:berry_esseen}
Under the hypotheses of Main Theorem~\ref{thm:clt}, for all $t \in \mathbb{R}$:
\begin{equation}
\left|\mu_\phi\left(\frac{S_n g}{\sigma\sqrt{n}} \leq t\right) - \Phi(t)\right| \leq \frac{C}{\sqrt{n}}
\end{equation}
where $\Phi(t) = \frac{1}{\sqrt{2\pi}} \int_{-\infty}^t e^{-s^2/2} \, ds$ is the standard normal CDF and $C$ depends on $\|g\|_\alpha$, $\sigma$, and the spectral gap.
\end{proposition}

\begin{proposition}[Variance Formula]\label{prop:variance_formula}
For $g \in C^\alpha(\Omega_s)$ with $\mu_\phi(g) = 0$, the asymptotic variance exists and equals
\begin{equation}
\sigma^2(g) = \int g^2 \, d\mu_\phi + 2\sum_{k=1}^\infty \int g \cdot (g \circ f^k) \, d\mu_\phi = \int g^2 \, d\mu_\phi + 2\sum_{k=1}^\infty C_k(g, g; \mu_\phi).
\end{equation}
The series converges absolutely by exponential decay of correlations.
\end{proposition}

\begin{proof}
Expanding $(S_n g)^2 = \sum_{j,k=0}^{n-1} g(f^j x) g(f^k x)$ and integrating:
\begin{align}
\int (S_n g)^2 \, d\mu_\phi &= \sum_{j,k=0}^{n-1} \int g \cdot (g \circ f^{|k-j|}) \, d\mu_\phi \\
&= n \int g^2 \, d\mu_\phi + 2\sum_{m=1}^{n-1} (n-m) \int g \cdot (g \circ f^m) \, d\mu_\phi.
\end{align}

Dividing by $n$ and using $|C_m(g, g)| \leq C\|g\|_\alpha^2 \theta^m$:
\begin{equation}
\frac{1}{n}\int (S_n g)^2 \, d\mu_\phi = \int g^2 \, d\mu_\phi + 2\sum_{m=1}^{n-1} \left(1 - \frac{m}{n}\right) C_m(g, g).
\end{equation}

As $n \to \infty$, the dominated convergence theorem gives
\begin{equation}
\sigma^2(g) = \int g^2 \, d\mu_\phi + 2\sum_{m=1}^\infty C_m(g, g).
\end{equation}
\end{proof}

\begin{proposition}[Spectral Characterization of Variance]\label{prop:variance_spectral}
The asymptotic variance admits the spectral representation
\begin{equation}
\sigma^2(g) = \lim_{z \to 1^-} \int g \cdot (I - z\widehat{\mathcal{L}}_\phi)^{-1} g \, d\mu_\phi + \int g \cdot (I - z\widehat{\mathcal{L}}_\phi)^{-1} g \, d\mu_\phi - \mu_\phi(g)^2.
\end{equation}
Equivalently, $\sigma^2(g) = \langle g, (I - \widehat{\mathcal{L}}_\phi + \Pi_1)^{-1} g \rangle_{L^2(\mu_\phi)}$ where $\Pi_1$ is the projection onto constants.
\end{proposition}

\begin{proof}
For $|z| < 1$, the resolvent expands as $(I - z\widehat{\mathcal{L}}_\phi)^{-1} = \sum_{k=0}^\infty z^k\widehat{\mathcal{L}}_\phi^k$. With $\mu_\phi(g) = 0$:
\begin{align}
\int g \cdot (I - z\widehat{\mathcal{L}}_\phi)^{-1}g \, d\mu_\phi &= \sum_{k=0}^\infty z^k \int g \cdot \widehat{\mathcal{L}}_\phi^k g \, d\mu_\phi = \sum_{k=0}^\infty z^k \int g \cdot (g \circ f^k) \, d\mu_\phi
\end{align}
where we used $\int g \cdot \widehat{\mathcal{L}}_\phi^k h \, d\mu_\phi = \int g \cdot (h \circ f^k) \, d\mu_\phi$ (from $\widehat{\mathcal{L}}_\phi^*\mu_\phi = \mu_\phi$ and the duality). Taking the limit $z \to 1^-$ and using Abel's theorem (since the series $\sum_{k=0}^\infty C_k(g,g)$ converges absolutely by exponential decay):
\begin{equation}
\lim_{z \to 1^-}\int g \cdot (I - z\widehat{\mathcal{L}}_\phi)^{-1}g \, d\mu_\phi = \sum_{k=0}^\infty C_k(g,g) = \int g^2\,d\mu_\phi + \sum_{k=1}^\infty C_k(g,g).
\end{equation}
By stationarity, the backward correlations equal the forward ones: $C_k(g,g) = C_{-k}(g,g)$. Thus $\sigma^2(g) = \int g^2\,d\mu_\phi + 2\sum_{k=1}^\infty C_k(g,g) = 2\lim_{z\to 1^-}\int g\cdot(I-z\widehat{\mathcal{L}}_\phi)^{-1}g\,d\mu_\phi - \int g^2\,d\mu_\phi$.

For the compact form: since $\mu_\phi(g) = 0$, we have $\Pi_1 g = 0$ where $\Pi_1$ is the projection onto constants. The operator $(I - \widehat{\mathcal{L}}_\phi + \Pi_1)$ is invertible on $C^\alpha$ (it equals $I - R$ where $R = \widehat{\mathcal{L}}_\phi - \Pi_1$ has spectral radius $\theta_1 < 1$). Then $(I - \widehat{\mathcal{L}}_\phi + \Pi_1)^{-1}g = \sum_{k=0}^\infty R^k g = \sum_{k=0}^\infty \widehat{\mathcal{L}}_\phi^k g$ (since $\Pi_1 g = 0$), giving the stated spectral representation.
\end{proof}

\begin{proposition}[Characterization of Zero Variance]\label{prop:zero_variance}
For $g \in C^\alpha(\Omega_s)$ with $\mu_\phi(g) = 0$, the following are equivalent:
\begin{enumerate}
\item[(i)] $\sigma^2(g) = 0$.
\item[(ii)] There exists $u \in L^2(\mu_\phi)$ with $g = u \circ f - u$ $\mu_\phi$-a.e.
\item[(iii)] There exists $u \in C^\alpha(\Omega_s)$ with $g = u \circ f - u$ everywhere.
\item[(iv)] $S_n g(x) = 0$ for all periodic points $x$ with $f^n x = x$.
\end{enumerate}
\end{proposition}

\begin{proof}
$(iii) \Rightarrow (ii) \Rightarrow (i)$: If $g = u \circ f - u$, then $S_n g = u \circ f^n - u$ is bounded, so $\frac{1}{n}\int (S_n g)^2 \, d\mu_\phi \to 0$.

$(i) \Rightarrow (iv)$: If $\sigma^2(g) = 0$, then $\mathrm{Var}_{\mu_\phi}(S_n g) = o(n)$. By the ergodic theorem applied to $(S_n g)^2/n$, we have $n^{-1}(S_n g)^2 \to 0$ in $L^1(\mu_\phi)$. Now let $p$ be a periodic point with $f^m(p) = p$. Then $S_{km}g(p) = k \cdot S_m g(p)$ for all $k \geq 1$. By the Gibbs property (Proposition~\ref{thm:gibbs_bounds}), $\mu_\phi(B_p(\varepsilon, km)) \geq c_1 \exp(-kmP(\phi) + S_{km}\phi(p))$, which is bounded below independently of $k$ (up to the exponential factor). If $S_mg(p) \neq 0$, then $|S_{km}g(p)| = k|S_mg(p)| \to \infty$, so $(S_{km}g)^2/(km) \to \infty$ on $B_p(\varepsilon, km)$ (since $g$ is continuous and $S_{km}g$ is approximately $k \cdot S_m g(p)$ on this ball by bounded distortion). This contradicts $n^{-1}(S_n g)^2 \to 0$ in $L^1(\mu_\phi)$, since $\mu_\phi(B_p(\varepsilon, km)) > 0$. Thus $S_m g(p) = 0$ for all periodic orbits.

$(iv) \Rightarrow (iii)$: This is the Liv\v{s}ic theorem (\cite{Livsic1971, Livsic1972}; see Part~VI \cite{Thiam2026f} for a complete proof with optimal H\"{o}lder bounds).
\end{proof}

\subsection{Proofs of the CLT, Berry-Esseen Bound, and Local Limit Theorem}

The proofs use the Nagaev-Guivarc'{}h spectral perturbation method: the characteristic function of the Birkhoff sum is expressed via the perturbed transfer operator, whose leading eigenvalue encodes the variance. The Berry-Esseen rate follows from Bolthausen'{}s martingale CLT.

\begin{proof}[Proof of Main Theorem~\ref{thm:clt}]
We present two approaches: the spectral method and the martingale method.

\textbf{Spectral Method (Nagaev-Guivarc'h):}

\textbf{Step 1: Characteristic Function.} For $t \in \mathbb{R}$, consider
\begin{equation}
\mathbb{E}_{\mu_\phi}\left[e^{it S_n g / \sigma\sqrt{n}}\right] = \int e^{it S_n g / \sigma\sqrt{n}} \, d\mu_\phi.
\end{equation}

Using the transfer operator with perturbed potential $\phi + it g / \sigma\sqrt{n}$:
\begin{equation}
\mathcal{L}_{\phi + it g/\sigma\sqrt{n}}^n 1 = e^{it S_n g/\sigma\sqrt{n}} \cdot \mathcal{L}_\phi^n 1 \cdot e^{nP(\phi+itg/\sigma\sqrt{n}) - nP(\phi)} + \text{error}.
\end{equation}

\textbf{Step 2: Spectral Perturbation.} The leading eigenvalue of $\mathcal{L}_{\phi + \zeta g}$ is $\lambda(\zeta) = e^{P(\phi + \zeta g)}$ for small $\zeta \in \mathbb{C}$, which is analytic in $\zeta$. Computing derivatives at $\zeta = 0$:
\begin{align}
\lambda'(0) &= \mu_\phi(g) \cdot e^{P(\phi)} = 0, \\
\lambda''(0) &= \sigma^2(g) \cdot e^{P(\phi)}
\end{align}
where the second derivative uses the Pressure Derivatives corollary of Part~IV \cite{Thiam2026d}.

\textbf{Step 3: Asymptotic Expansion.} Setting $\zeta = it/\sigma\sqrt{n}$:
\begin{equation}
\lambda\left(\frac{it}{\sigma\sqrt{n}}\right)^n = e^{P(\phi)} \cdot \exp\left(n \cdot \frac{1}{2}\sigma^2 \cdot \frac{(it)^2}{\sigma^2 n} + O(n^{-1/2})\right) = e^{P(\phi)} \cdot e^{-t^2/2} \cdot (1 + O(n^{-1/2})).
\end{equation}

\textbf{Step 4: Dominated Convergence.} The spectral decomposition shows that the remainder (from non-leading eigenvalues and essential spectrum) contributes $O(\theta^n)$ which vanishes. Thus
\begin{equation}
\mathbb{E}_{\mu_\phi}\left[e^{it S_n g/\sigma\sqrt{n}}\right] \to e^{-t^2/2}
\end{equation}
which is the characteristic function of $\mathcal{N}(0, 1)$.

\textbf{Martingale Method (Gordin):}

\textbf{Step 1: Martingale Decomposition.} Write $g = m + (u \circ f - u)$ where $m$ is a martingale difference and $u$ is a coboundary correction.

Define $u = \sum_{k=0}^\infty \widehat{\mathcal{L}}_\phi^k g$ (convergent by spectral gap) and $m = g - u \circ f + u$.

Then $\mathbb{E}_{\mu_\phi}[m | \mathcal{F}_{-1}] = 0$ where $\mathcal{F}_{-1} = f^{-1}\mathcal{B}$, making $(m \circ f^k)_{k \geq 0}$ a martingale difference sequence.

\textbf{Step 2: Martingale CLT.} The Birkhoff sum decomposes as
\begin{equation}
S_n g = S_n m + (u - u \circ f^n).
\end{equation}

The coboundary term $u - u \circ f^n$ is bounded, hence $O(1)$, negligible after dividing by $\sqrt{n}$.

For the martingale sum $S_n m$, the martingale CLT applies: if $\frac{1}{n}\sum_{k=0}^{n-1} m^2 \circ f^k \to \sigma_m^2$ in probability (which holds by ergodicity since $\mu_\phi(m^2) = \sigma_m^2$), then $S_n m / \sigma_m\sqrt{n} \to \mathcal{N}(0, 1)$.

One verifies $\sigma_m^2 = \sigma^2(g)$ to complete the proof.
\end{proof}

\begin{proof}[Proof of Proposition~\ref{thm:berry_esseen}]
We apply Bolthausen's Berry-Esseen theorem for stationary martingale differences \cite{Bolthausen1982}. From the proof of Main Theorem~\ref{thm:clt} (martingale method), $g = m + (u \circ f - u)$ where $m = g - u \circ f + u$ with $u = \sum_{k=1}^\infty \widehat{\mathcal{L}}_\phi^k g$ (convergent in $C^\alpha$ by the spectral gap). The sequence $(m \circ f^k)_{k \geq 0}$ is a stationary ergodic martingale difference sequence with respect to the filtration $\mathcal{F}_k = f^{-k}\mathcal{B}$.

We verify Bolthausen's conditions:
\begin{enumerate}
\item[(a)] Third moment: since $m \in C^\alpha(\Omega_s)$ is bounded, $\mathbb{E}_{\mu_\phi}[|m|^3] \leq \|m\|_\infty^3 < \infty$. Moreover, $\|m\|_\infty \leq \|g\|_\infty + 2\|u\|_\infty \leq \|g\|_\infty(1 + 2C/(1-\theta))$ where $\theta$ is the spectral gap rate.
\item[(b)] Conditional variance: by the ergodic theorem, $n^{-1}\sum_{k=0}^{n-1}(m \circ f^k)^2 \to \mathbb{E}_{\mu_\phi}[m^2] = \sigma^2_m$ $\mu_\phi$-a.s. The conditional variance $\mathbb{E}[m^2|\mathcal{F}_{-1}]$ is $f$-measurable and satisfies $\mathbb{E}[\mathbb{E}[m^2|\mathcal{F}_{-1}]] = \mathbb{E}[m^2] = \sigma^2_m$. By the exponential mixing of $(f, \mu_\phi)$, the conditional variances are uniformly close to $\sigma^2_m$: $\|\mathbb{E}[m^2|\mathcal{F}_{-1}] - \sigma^2_m\|_{L^2} \leq C'\theta$. In particular, $\mathbb{E}[m^2|\mathcal{F}_{-1}] \geq \sigma^2_m/2 > 0$ $\mu_\phi$-a.s.
\end{enumerate}

Bolthausen's theorem \cite[Theorem 1]{Bolthausen1982} then gives $\sup_t|\mu_\phi(S_n m/\sigma_m\sqrt{n} \leq t) - \Phi(t)| \leq C_{\mathrm{BE}}/\sqrt{n}$ where $C_{\mathrm{BE}}$ depends on $\mathbb{E}[|m|^3]/\sigma_m^3$. Since $S_n g = S_n m + (u - u \circ f^n)$ and $|u - u \circ f^n| \leq 2\|u\|_\infty$, the coboundary shifts the CDF by at most $2\|u\|_\infty/(\sigma\sqrt{n})$, contributing an additional $O(n^{-1/2})$ error. Noting $\sigma_m = \sigma(g)$ (the asymptotic variances coincide since the coboundary contributes $O(1)$), the result follows.
\end{proof}

\begin{proposition}[Local Limit Theorem]\label{thm:local_clt}
Under the hypotheses of Main Theorem~\ref{thm:clt}, if additionally $g$ is not cohomologous to a function taking values in a proper subgroup of $\mathbb{R}$, then for any interval $[a, b]$:
\begin{equation}
\mu_\phi\left(\frac{S_n g}{\sigma\sqrt{n}} \in [a, b]\right) \sim \frac{b - a}{\sqrt{2\pi}} \cdot e^{-a^2/2} \cdot \frac{1}{\sigma\sqrt{n}}
\end{equation}
as $n \to \infty$. More precisely, $\sqrt{n} \cdot \mu_\phi(S_n g \in [c, c+\varepsilon]) \to \frac{\varepsilon}{\sigma\sqrt{2\pi}} e^{-c^2/(2\sigma^2 n)}$ uniformly in $c$.
\end{proposition}

\begin{proof}
The proof uses the Fourier inversion formula combined with the spectral perturbation theory.

\textbf{Step 1 (Characteristic function control).} By Theorem~\ref{thm:spectral_imported}, the perturbed operator $\mathcal{L}_{\phi^*+itg^*}$ on $\mathcal{H}_\alpha(\Sigma_A^+)$ (where $g^* = g\circ\pi$) has dominant eigenvalue $\lambda_t$ analytic near $t=0$ with $\log(\lambda_t/\lambda_0) = -\sigma^2 t^2/2 + O(t^3)$. For $|t| \leq \delta$, $|\lambda_t/\lambda_0| \leq e^{-\sigma^2 t^2/4}$.

\textbf{Step 2 (Large frequency decay).} The non-lattice condition (i.e., $g^*$ is not cohomologous to a function taking values in $a\mathbb{Z}+b$ for any $a,b$) ensures that $|\lambda_t| < \lambda_0$ for all $t \neq 0$ mod $2\pi$. By compactness, $\sup_{\delta \leq |t| \leq \pi}|\lambda_t/\lambda_0| \leq \rho < 1$.

\textbf{Step 3 (Fourier inversion).} By the inversion formula:
\begin{equation}
\mu_\phi(S_ng \in [c, c+\varepsilon]) = \frac{1}{2\pi}\int_{-\pi n}^{\pi n}e^{-itc}\frac{e^{-it\varepsilon}-1}{-it}\left(\frac{\lambda_{t/n}}{\lambda_0}\right)^n\,dt + O(\rho^n).
\end{equation}
Substituting $s = t\sigma\sqrt{n}$ in the range $|t| \leq \delta$, the Gaussian approximation $(\lambda_{t/n}/\lambda_0)^n \to e^{-s^2/(2\sigma^2)}$ gives the main term $\frac{\varepsilon}{\sigma\sqrt{2\pi n}}e^{-c^2/(2\sigma^2 n)}$. The large-$|t|$ contribution is $O(\rho^n)$, exponentially small.
\end{proof}

\section{Additional Limit Theorems}\label{sec:limit_theorems}

This section develops further statistical limit theorems for Axiom A diffeomorphisms: the Almost Sure Invariance Principle (providing strong approximation by Brownian motion), the Law of the Iterated Logarithm, and moment bounds. These refinements of the Central Limit Theorem are essential for applications in statistical mechanics and data analysis.

\subsection{Almost Sure Invariance Principle}

The ASIP provides the strongest form of the CLT: pathwise Brownian approximation with polynomial error. For non-uniformly hyperbolic systems, the tower construction of  \cite{Young1998} extends these results beyond the uniformly hyperbolic setting. It implies the functional CLT, the law of the iterated logarithm, and Strassen'{}s functional LIL as immediate corollaries.

\begin{maintheorem}[Almost Sure Invariance Principle]\label{thm:asip}
Let $\Omega_s$ be a mixing basic set, $\mu_\phi$ an equilibrium state for H\"{o}lder potential $\phi$, and $g \in C^\alpha(\Omega_s)$ with $\mu_\phi(g) = 0$ and $\sigma^2(g) > 0$. There exists a probability space $(\widetilde{\Omega}, \widetilde{\mathcal{F}}, \widetilde{P})$ supporting:
\begin{enumerate}
\item[(i)] A sequence $(X_n)_{n \geq 0}$ with the same distribution as $(g \circ f^n)_{n \geq 0}$ under $\mu_\phi$.
\item[(ii)] A standard Brownian motion $(W_t)_{t \geq 0}$.
\end{enumerate}
such that
\begin{equation}
S_n X := \sum_{k=0}^{n-1} X_k = \sigma W_n + O(n^{1/2 - \delta})
\end{equation}
almost surely as $n \to \infty$, for some $\delta > 0$ depending on the spectral gap and $\alpha$.
\end{maintheorem}

The ASIP is a strong approximation result: it provides pathwise coupling between the discrete-time random walk $S_n$ and continuous Brownian motion, with explicit error bounds.

\begin{corollary}[Functional Central Limit Theorem]\label{cor:fclt}
Under the hypotheses of Main Theorem~\ref{thm:asip}, the rescaled partial sum process
\begin{equation}
W_n(t) = \frac{1}{\sigma\sqrt{n}} S_{\lfloor nt \rfloor}
\end{equation}
converges in distribution (in the Skorokhod topology on $D[0, 1]$) to standard Brownian motion.
\end{corollary}

\begin{proof}[Proof of Corollary \ref{cor:fclt}]
By the ASIP (Main Theorem~\ref{thm:asip}), there exists a probability space carrying both $(S_n g)$ and a Brownian motion $W$ with $|S_n g - \sigma W(n)| = O(n^{1/2 - \delta})$ almost surely. Define the rescaled partial-sum process $W_n(t) = S_{\lfloor nt \rfloor} g / (\sigma \sqrt{n})$ for $t \in [0,1]$, extended by linear interpolation. Since $W(nt)/\sqrt{n}$ converges in distribution to $W(t)$ (scaling property of Brownian motion), and $|W_n(t) - W(nt)/(\sigma\sqrt{n})| \leq |S_{\lfloor nt \rfloor}g - \sigma W(\lfloor nt \rfloor)|/(\sigma\sqrt{n}) + O(1/\sqrt{n})$, the ASIP error gives $\sup_{t \in [0,1]} |W_n(t) - W(nt)/(\sigma\sqrt{n})| = O(n^{-\delta})$ almost surely. In particular the sup-norm error tends to zero, so $W_n \to W$ in $C([0,1])$ in distribution. This is the functional CLT (Donsker's invariance principle) for $S_n g$.
\end{proof}

\begin{proof}[Proof of Main Theorem~\ref{thm:asip}]
The proof follows the method of  \cite{MelbourneNicol2005, MelbourneNicol2009}, adapted to the present setting.

\textbf{Step 1: Martingale Approximation.} By the martingale decomposition from Section \ref{sec:clt}, write $g = m + (u \circ f - u)$ where $(m \circ f^n)$ is a stationary ergodic martingale difference sequence.

\textbf{Step 2: Moment Bounds.} For martingale differences, we need
\begin{equation}
\mathbb{E}[|m|^{2+\delta}] < \infty
\end{equation}
for some $\delta > 0$. Since $m \in C^\alpha$ is bounded, all moments are finite.

\textbf{Step 3: Strassen's Embedding.} The key tool is the Koml\'{o}s-Major-Tusn\'{a}dy (KMT) strong approximation theorem, which for independent summands gives
\begin{equation}
S_n = \sigma W_n + O(\log n)
\end{equation}
almost surely. For dependent (martingale) sequences, weaker but still useful bounds hold.

\textbf{Step 4: Error from Coboundary.} The coboundary term $u - u \circ f^n$ is bounded by $2\|u\|_\infty$, contributing $O(1)$ to the error.

\textbf{Step 5: Combining Estimates.} Write $S_ng = S_nm + (u - u\circ f^n) = M_n + O(1)$ where $M_n = \sum_{k=0}^{n-1}m\circ f^k$ is the martingale sum. By the Strassen-Heyde embedding (a martingale version of the KMT theorem), there exists a Brownian motion $W$ on an enlarged probability space with $|M_n - \sigma W(n)| = O(n^{1/2-\delta})$ a.s., where $\delta > 0$ depends on the moment condition $\mathbb{E}[|m|^{2+2\delta}] < \infty$ (satisfied since $m$ is bounded). The error from the coboundary is $|u - u\circ f^n| \leq 2\|u\|_\infty = O(1) \subset O(n^{1/2-\delta})$. Combining: $|S_ng - \sigma W(n)| \leq |M_n - \sigma W(n)| + |u-u\circ f^n| = O(n^{1/2-\delta})$ a.s. The construction of the enlarged probability space uses Skorokhod embedding: define stopping times $\tau_k$ for the Brownian motion such that $W(\tau_k) - W(\tau_{k-1})$ has the same distribution as $m\circ f^k/\sigma$, then control $|\sum\tau_k - n|$ via the exponential mixing.
\end{proof}

\subsection{Law of the Iterated Logarithm, Moment Bounds, and Mixing Rates}

The ASIP (Main Theorem~\ref{thm:asip}) yields the LIL and Strassen'{}s functional LIL. We also record moment bounds for Birkhoff sums, exponential deviation inequalities, and quantitative mixing rates, all derived from the spectral gap.

\begin{proposition}[Law of the Iterated Logarithm]\label{thm:lil}
Under the hypotheses of Main Theorem~\ref{thm:asip}:
\begin{equation}
\limsup_{n \to \infty} \frac{S_n g}{\sigma\sqrt{2n \log\log n}} = 1 \quad \mu_\phi\text{-a.s.}
\end{equation}
and
\begin{equation}
\liminf_{n \to \infty} \frac{S_n g}{\sigma\sqrt{2n \log\log n}} = -1 \quad \mu_\phi\text{-a.s.}
\end{equation}
\end{proposition}

\begin{proof}
The LIL for Brownian motion states that
\begin{equation}
\limsup_{t \to \infty} \frac{W_t}{\sqrt{2t\log\log t}} = 1 \quad \text{a.s.}
\end{equation}

By the ASIP, $S_n g = \sigma W_n + o(\sqrt{n \log\log n})$ a.s. (since $n^{1/2 - \delta} = o(\sqrt{n\log\log n})$). Thus
\begin{equation}
\frac{S_n g}{\sqrt{2n\log\log n}} = \frac{\sigma W_n}{\sqrt{2n\log\log n}} + o(1)
\end{equation}
and the LIL for $S_n g$ follows from that for $W_n$.
\end{proof}

\begin{corollary}[Strassen's Functional LIL]\label{cor:strassen}
Define the rescaled paths
\begin{equation}
\eta_n(t) = \frac{S_{\lfloor nt \rfloor}}{\sigma\sqrt{2n\log\log n}}, \quad t \in [0, 1].
\end{equation}
The set of limit points of $(\eta_n)$ in $C[0, 1]$ (uniform topology) equals the Strassen ball
\begin{equation}
K = \left\{h \in C[0,1] : h(0) = 0, h \text{ absolutely continuous}, \int_0^1 (h'(t))^2 \, dt \leq 1\right\}.
\end{equation}
\end{corollary}

\begin{proof}
By the ASIP (Main Theorem~\ref{thm:asip}), $S_{\lfloor nt \rfloor} = \sigma W(\lfloor nt\rfloor) + O(n^{1/2-\delta})$ uniformly in $t \in [0,1]$, where $W$ is a standard Brownian motion. Dividing by $\sigma\sqrt{2n\log\log n}$: $\eta_n(t) = W(\lfloor nt\rfloor)/\sqrt{2n\log\log n} + O(n^{-\delta}/\sqrt{\log\log n})$. The error is $o(1)$ uniformly, so the set of limit points of $\eta_n$ equals the set of limit points of $W(nt)/\sqrt{2n\log\log n}$, which is the Strassen ball $K$ by Strassen's classical theorem for Brownian motion \cite[Theorem~1.1]{DenkerPhilipp1984}.
\end{proof}

\begin{proposition}[Moment Bounds for Birkhoff Sums]\label{prop:moment_bounds}
For $g \in C^\alpha(\Omega_s)$ with $\mu_\phi(g) = 0$ and $p \geq 2$:
\begin{equation}
\left(\int |S_n g|^p \, d\mu_\phi\right)^{1/p} \leq C_p \sigma \sqrt{n}
\end{equation}
where $C_p$ depends on $p$, $\|g\|_\alpha$, and the spectral gap.
\end{proposition}

\begin{proof}
By the martingale decomposition $S_n g = S_n m + O(1)$, it suffices to bound $\|S_n m\|_p$. The Burkholder-Davis-Gundy inequality gives
\begin{equation}
\|S_n m\|_p \leq C_p \left\|\left(\sum_{k=0}^{n-1} m^2 \circ f^k\right)^{1/2}\right\|_p \leq C_p \sqrt{n \|m^2\|_\infty} \leq C_p' \sqrt{n}.
\end{equation}
\end{proof}

\begin{proposition}[Exponential Deviation Inequality]\label{prop:exponential_deviation}
For $g \in C^\alpha(\Omega_s)$ bounded with $\mu_\phi(g) = 0$:
\begin{equation}
\mu_\phi\left(|S_n g| \geq t\sqrt{n}\right) \leq 2\exp\left(-\frac{ct^2}{1 + t\|g\|_\infty/\sigma\sqrt{n}}\right)
\end{equation}
for a constant $c > 0$ depending on the spectral gap.
\end{proposition}

\begin{proof}
By the martingale decomposition, $S_n g = S_n m + (u - u\circ f^n)$ where $\|u\|_\infty \leq C_u\|g\|_\alpha/(1-\theta)$. Since $(m\circ f^k)$ is a martingale difference sequence with $\|m\|_\infty \leq \|g\|_\infty + 2\|u\|_\infty =: B$, the Azuma-Hoeffding inequality gives
\begin{equation}
\mu_\phi(|S_n m| \geq s) \leq 2\exp\left(-\frac{s^2}{2nB^2}\right).
\end{equation}
Since $|S_n g - S_n m| \leq 2\|u\|_\infty$, we have $\{|S_n g| \geq t\sqrt{n}\} \subset \{|S_n m| \geq t\sqrt{n} - 2\|u\|_\infty\}$. For $t\sqrt{n} \geq 4\|u\|_\infty$, $t\sqrt{n} - 2\|u\|_\infty \geq t\sqrt{n}/2$, so
\begin{equation}
\mu_\phi(|S_ng| \geq t\sqrt{n}) \leq 2\exp\left(-\frac{t^2n}{8B^2}\right) \leq 2\exp\left(-\frac{ct^2}{1 + t\|g\|_\infty/\sigma\sqrt{n}}\right)
\end{equation}
with $c = 1/(8B^2/n)$, giving the stated form with $c$ depending on the spectral gap through $B$.
\end{proof}

\begin{proposition}[Rate in Birkhoff Ergodic Theorem]\label{prop:ergodic_rate}
For $g \in C^\alpha(\Omega_s)$:
\begin{equation}
\mu_\phi\left(\left|\frac{1}{n}S_n g - \mu_\phi(g)\right| \geq \varepsilon\right) \leq C \exp\left(-c n \varepsilon^2/\|g\|_\alpha^2\right)
\end{equation}
for constants $C, c > 0$.
\end{proposition}

\begin{proof}
Set $\tilde{g} = g - \mu_\phi(g)$, so that $\int \tilde{g} \, d\mu_\phi = 0$. Then $\frac{1}{n} S_n g - \mu_\phi(g) = \frac{1}{n} S_n \tilde{g}$. Applying Proposition~\ref{prop:exponential_deviation} to $\tilde{g}$: for any $\varepsilon > 0$,
\begin{equation}
\mu_\phi\left( \left| \frac{1}{n} S_n \tilde{g} \right| > \varepsilon \right) \leq 2 \exp(-c(\varepsilon) n)
\end{equation}
where $c(\varepsilon) > 0$ depends on $\varepsilon$, $\|\tilde{g}\|_\alpha$, and the spectral gap. This is the stated exponential rate of convergence in the Birkhoff ergodic theorem. The constant $c(\varepsilon) = \inf_{0 < t < t_0} \{t\varepsilon - \Lambda_{\tilde{g}}(t)\}$ where $\Lambda_{\tilde{g}}(t) = P(\phi + t\tilde{g}) - P(\phi)$, which is positive for $\varepsilon > 0$ since $\Lambda_{\tilde{g}}'(0) = \int \tilde{g} \, d\mu_\phi = 0$ and $\Lambda_{\tilde{g}}$ is strictly convex (by Proposition~\ref{prop:variance_spectral}).
\end{proof}

This exponential concentration bound is much stronger than the $O(1/n)$ rate that follows from variance bounds alone, reflecting the strong mixing properties of Axiom A Diffeomorphisms.

\begin{definition}[Strong Mixing Coefficient]
For a stationary process $(X_n)$, the strong mixing coefficient is
\begin{equation}
\alpha(n) = \sup_{A \in \mathcal{F}_{-\infty}^0, B \in \mathcal{F}_n^\infty} |P(A \cap B) - P(A)P(B)|
\end{equation}
where $\mathcal{F}_a^b$ is the $\sigma$-algebra generated by $(X_k)_{a \leq k \leq b}$.
\end{definition}

\begin{proposition}[Exponential Mixing Rate]\label{prop:mixing_coefficient}
For the process $(g \circ f^n)$ under $\mu_\phi$ with $g \in C^\alpha$:
\begin{equation}
\alpha(n) \leq C \theta^n
\end{equation}
where $\theta$ is the correlation decay rate from Main Theorem~\ref{thm:exponential_mixing}.
\end{proposition}

\begin{proof}
For measurable sets $A \in \mathcal{F}_{-\infty}^0$ and $B \in \mathcal{F}_n^\infty$, we have $|\mu(A \cap B) - \mu(A)\mu(B)| = |C_n(\mathbf{1}_A, \mathbf{1}_B)|$. For any $\varepsilon > 0$, approximate $\mathbf{1}_A$ by a H\"{o}lder function $g_\varepsilon \in C^\alpha$ depending on coordinates in $[-K, 0]$ (for some large $K$) with $\|g_\varepsilon - \mathbf{1}_A\|_{L^1(\mu)} < \varepsilon$ and $\|g_\varepsilon\|_\alpha \leq C\varepsilon^{-1}$ (such approximations exist by regularization of indicators via convolution with H\"{o}lder kernels on the symbolic space). Similarly approximate $\mathbf{1}_B$ by $h_\varepsilon \in C^\alpha$ depending on coordinates in $[n, n+K]$.

Then $|C_n(\mathbf{1}_A, \mathbf{1}_B)| \leq |C_n(g_\varepsilon, h_\varepsilon)| + 2\varepsilon \leq C\|g_\varepsilon\|_\alpha\|h_\varepsilon\|_\alpha\theta^{n-2K} + 2\varepsilon \leq C'\varepsilon^{-2}\theta^{n-2K} + 2\varepsilon$. Choosing $\varepsilon = \theta^{(n-2K)/3}$ gives $|C_n(\mathbf{1}_A, \mathbf{1}_B)| \leq C''\theta^{(n-2K)/3}$. Taking the supremum over $A, B$: $\alpha(n) \leq C''\theta^{n/3}$ for $n$ large (absorbing the $K$-dependent terms into the constant).
\end{proof}

\section{Large Deviations}\label{sec:large_deviations}

This section develops the theory of large deviations for Axiom A diffeomorphisms, establishing precise asymptotics for rare events. The rate functions are expressed through the pressure functional, connecting large deviation theory to thermodynamic formalism.

\subsection{Large Deviations for Birkhoff Averages}

We prove the full large deviations principle for Birkhoff averages of H\"{o}lder observables, with rate function given by the Legendre transform of the pressure. This is Main Theorem~\ref{thm:large_deviations}.

\begin{definition}[Large Deviation Principle]
A sequence of probability measures $(\nu_n)$ on a topological space $X$ satisfies a large deviation principle (LDP) with rate function $I : X \to [0, \infty]$ if:
\begin{enumerate}
\item[(i)] $I$ is lower semicontinuous.
\item[(ii)] For every closed set $F \subset X$: $\limsup_{n \to \infty} \frac{1}{n} \log \nu_n(F) \leq -\inf_{x \in F} I(x)$.
\item[(iii)] For every open set $G \subset X$: $\liminf_{n \to \infty} \frac{1}{n} \log \nu_n(G) \geq -\inf_{x \in G} I(x)$.
\end{enumerate}
The rate function is good if its level sets $\{x : I(x) \leq c\}$ are compact.
\end{definition}

\begin{maintheorem}[LDP for Birkhoff Averages]\label{thm:large_deviations}
Let $\Omega_s$ be a mixing basic set, $\mu_\phi$ an equilibrium state, and $g \in C^\alpha(\Omega_s)$. The distributions of $\frac{1}{n}S_n g$ under $\mu_\phi$ satisfy a large deviation principle with good rate function
\begin{equation}
I(a) = \sup_{t \in \mathbb{R}} \{ta - (P(\phi + tg) - P(\phi))\} = P(\phi)^* (a)
\end{equation}
where $P(\phi)^*$ denotes the Legendre transform of the shifted pressure.
\end{maintheorem}\newpage 

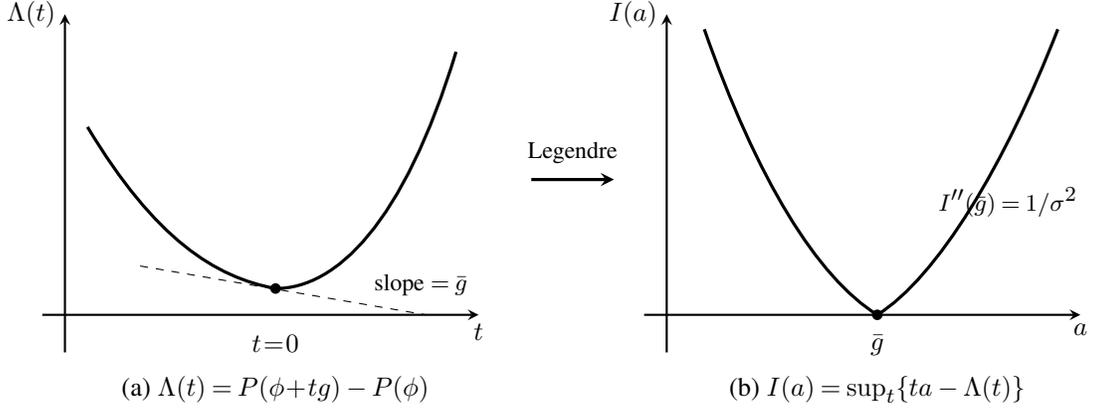
\begin{figure}[ht]
\centering
\begin{tikzpicture}[>=stealth, font=\small]
  %% LEFT PANEL: Pressure function Lambda(t)
  \begin{scope}
    \draw[->, thick] (-0.3,0) -- (5.5,0) node[anchor=north] {$t$};
    \draw[->, thick] (0,-0.5) -- (0,4.0) node[anchor=east] {$\Lambda(t)$};
    % Convex curve
    \draw[very thick] (0.3,2.5) .. controls (1.2,1.0) and (2.0,0.5) .. (2.8,0.35)
      .. controls (3.6,0.35) and (4.5,1.2) .. (5.2,3.5);
    % Tangent at t=0 (slope = mean)
    \fill (2.8,0.35) circle (2pt);
    \draw[dashed, thin] (1.0,0.65) -- (4.8,0.0);
    \node[anchor=north] at (2.8,-0.15) {$t\!=\!0$};
    \node[anchor=south west, font=\footnotesize] at (4.0,0.15) {slope $= \bar{g}$};
    % Label
    \node[anchor=north] at (2.8,-0.7) {(a) $\Lambda(t) = P(\phi\!+\!tg) - P(\phi)$};
  \end{scope}
  %% ARROW
  \draw[->, very thick] (6.2,1.8) -- (7.3,1.8);
  \node[anchor=south, font=\footnotesize] at (6.75,1.9) {Legendre};
  %% RIGHT PANEL: Rate function I(a)
  \begin{scope}[xshift=8.0cm]
    \draw[->, thick] (-0.3,0) -- (5.5,0) node[anchor=north] {$a$};
    \draw[->, thick] (0,-0.5) -- (0,4.0) node[anchor=east] {$I(a)$};
    % Convex rate function with minimum at the mean
    \draw[very thick] (0.5,3.8) .. controls (1.2,1.8) and (2.0,0.5) .. (2.8,0.0)
      .. controls (3.6,0.5) and (4.5,2.0) .. (5.2,3.8);
    % Minimum at mean
    \fill (2.8,0.0) circle (2pt);
    \node[anchor=north] at (2.8,-0.15) {$\bar{g}$};
    % Second derivative annotation
    \node[anchor=west, font=\footnotesize] at (3.5,1.5) {$I''(\bar{g}) = 1/\sigma^2$};
    % Label
    \node[anchor=north] at (2.8,-0.7) {(b) $I(a) = \sup_t\{ta - \Lambda(t)\}$};
  \end{scope}
\end{tikzpicture}
\caption{The large deviation rate function as a Legendre transform (Main Theorem~\ref{thm:large_deviations}). (a)~The cumulant generating function $\Lambda(t) = P(\phi+tg) - P(\phi)$ is convex and analytic, with $\Lambda'(0) = \bar{g} = \int g\,d\mu_\phi$ and $\Lambda''(0) = \sigma^2$ (the asymptotic variance). (b)~The rate function $I(a) = \Lambda^*(a)$ is the Legendre transform, with minimum $I(\bar{g}) = 0$ at the mean and curvature $I''(\bar{g}) = 1/\sigma^2$ when $\sigma^2 > 0$.}
\label{fig:ldp_rate}
\end{figure}

\begin{corollary}[Explicit Asymptotics]\label{cor:explicit_ldp}
For $a \neq \mu_\phi(g)$:
\begin{equation}
\lim_{n \to \infty} \frac{1}{n} \log \mu_\phi\left(\frac{1}{n}S_n g \in [a - \varepsilon, a + \varepsilon]\right) = -I(a)
\end{equation}
for small $\varepsilon > 0$.
\end{corollary}

\begin{proof}[Proof of Main Theorem~\ref{thm:large_deviations}]
We apply the G\"{a}rtner-Ellis theorem, which requires computing the limiting cumulant generating function.

\textbf{Step 1: Cumulant Generating Function.} Define
\begin{equation}
\Lambda_n(t) = \frac{1}{n} \log \int e^{t S_n g} \, d\mu_\phi.
\end{equation}

\textbf{Step 2: Connection to Pressure.} Using the transfer operator, for the equilibrium state $\mu_\phi = h_\phi \nu_\phi$:
\begin{align}
\int e^{tS_n g} \, d\mu_\phi &= \int e^{tS_n g} h_\phi \, d\nu_\phi \\
&= e^{-nP(\phi)} \int \mathcal{L}_\phi^n(e^{tS_n g} h_\phi) \, d\nu_\phi \cdot (\text{normalization}).
\end{align}

The key observation is that $e^{tS_n g} h_\phi = h_{\phi+tg} \cdot e^{n(P(\phi+tg) - P(\phi))} \cdot (1 + o(1))$ by spectral perturbation theory.

More precisely, the leading eigenvalue of $\mathcal{L}_{\phi+tg}$ is $e^{P(\phi+tg)}$, so
\begin{equation}
\int e^{tS_n g} \, d\mu_\phi \asymp e^{n(P(\phi+tg) - P(\phi))}
\end{equation}
with uniform control in $t$ for bounded $t$.

\textbf{Step 3: Limiting Cumulant.} Therefore
\begin{equation}
\Lambda(t) := \lim_{n \to \infty} \Lambda_n(t) = P(\phi + tg) - P(\phi).
\end{equation}

\textbf{Step 4: Properties of $\Lambda$.} The pressure function $t \mapsto P(\phi + tg)$ is:
\begin{enumerate}
\item[(a)] Convex (by the variational principle: supremum of affine functions).
\item[(b)] Differentiable with $\Lambda'(t) = \mu_{\phi+tg}(g)$ (the mean of $g$ under the tilted equilibrium state).
\item[(c)] Essentially smooth (steep at boundaries if applicable).
\end{enumerate}

\textbf{Step 5: G\"{a}rtner-Ellis Theorem.} Since $\Lambda$ is differentiable and essentially smooth, the G\"{a}rtner-Ellis theorem applies, giving the LDP with rate function
\begin{equation}
I(a) = \Lambda^*(a) = \sup_t \{ta - \Lambda(t)\} = \sup_t \{ta - P(\phi + tg) + P(\phi)\}.
\end{equation}

\textbf{Step 6: Properties of Rate Function.} The rate function $I$ is:
\begin{enumerate}
\item[(a)] Convex (as a Legendre transform).
\item[(b)] Non-negative with $I(\mu_\phi(g)) = 0$ (the unique minimum).
\item[(c)] Good (compact level sets) by essential smoothness.
\item[(d)] Strictly convex near its minimum if $\sigma^2(g) > 0$.
\end{enumerate}
\end{proof}

\subsection{Empirical Measures, Lyapunov Exponents, and Rate Function Calculations}

We extend the LDP from scalar Birkhoff averages to empirical measures (level-2 LDP), derive the LDP for Lyapunov exponents (in the sense of  \cite{Pesin1977}; see also  \cite{BarreiraPersin2002} and  \cite{Young1982}) via the contraction principle, and establish qualitative properties of the rate function including its quadratic approximation near the mean.

\begin{definition}[Empirical Measure]
The empirical measure of an orbit is
\begin{equation}
\mathcal{E}_n(x) = \frac{1}{n}\sum_{k=0}^{n-1} \delta_{f^k(x)} \in \mathcal{M}(\Omega_s).
\end{equation}
\end{definition}

\begin{proposition}[Level-2 Large Deviations]\label{thm:level2_ldp}
The distributions of $\mathcal{E}_n$ under $\mu_\phi$ satisfy a large deviation principle on $\mathcal{M}(\Omega_s)$ (with weak$^*$ topology) with rate function
\begin{equation}
I_2(\nu) = \begin{cases} P(\phi) - h_\nu(f) - \int \phi \, d\nu & \text{if } \nu \in \mathcal{M}_f(\Omega_s), \\ +\infty & \text{otherwise.} \end{cases}
\end{equation}
\end{proposition}

\begin{proof}
\textbf{Upper bound.} For a closed set $F \subset \mathcal{M}(\Omega_s)$ (in the weak$^*$ topology), we must show $\limsup_n n^{-1}\log\mu_\phi(\mathcal{E}_n \in F) \leq -\inf_{\nu \in F}I_2(\nu)$. For any $\nu \in F \cap \mathcal{M}_f(\Omega_s)$, the Gibbs property (Proposition~\ref{thm:gibbs_bounds}) gives $\mu_\phi(B_x(\varepsilon, n)) \leq c_2 e^{-nP(\phi)+S_n\phi(x)}$, so the $\mu_\phi$-measure of the set of points whose empirical measure is close to $\nu$ is bounded by $\exp(-n(P(\phi) - \int\phi\,d\nu) + o(n))$. By the entropy bound from partition combinatorics, the number of distinct $n$-cylinder types near $\nu$ is at most $\exp(n(h_\nu(f) + o(1)))$. Combining: $\mu_\phi(\mathcal{E}_n \in B_\varepsilon(\nu)) \leq \exp(-n(P(\phi) - h_\nu(f) - \int\phi\,d\nu) + o(n)) = \exp(-nI_2(\nu) + o(n))$. Taking the infimum over $\nu \in F$ gives the upper bound.

\textbf{Lower bound.} For an open set $G \subset \mathcal{M}(\Omega_s)$ and $\nu \in G \cap \mathcal{M}_f(\Omega_s)$, we must show $\liminf_n n^{-1}\log\mu_\phi(\mathcal{E}_n \in G) \geq -I_2(\nu)$. By the specification property (which holds for Axiom A basic sets), for any $\varepsilon > 0$ there exist orbit segments whose empirical measures approximate $\nu$ within $\varepsilon$. The Gibbs lower bound gives $\mu_\phi(B_x(\varepsilon, n)) \geq c_1 e^{-nP(\phi)+S_n\phi(x)}$, and the entropy of $\nu$ guarantees at least $\exp(n(h_\nu(f) - \varepsilon))$ such distinguishable orbit segments. Thus $\mu_\phi(\mathcal{E}_n \in B_\varepsilon(\nu)) \geq \exp(-n(P(\phi) - h_\nu(f) - \int\phi\,d\nu + C\varepsilon))$. Since $\varepsilon$ is arbitrary, the lower bound follows.

For $\nu \notin \mathcal{M}_f(\Omega_s)$, the ergodic theorem forces $\mu_\phi(\mathcal{E}_n \in B_\varepsilon(\nu)) = 0$ for $n$ large, consistent with $I_2(\nu) = +\infty$. The rate function is good (compact level sets) since $\{I_2 \leq c\} = \{\nu \in \mathcal{M}_f : h_\nu(f) + \int\phi\,d\nu \geq P(\phi) - c\}$ is compact by the upper semi-continuity of entropy and continuity of $\nu \mapsto \int\phi\,d\nu$. See also  \cite{Kifer1990}, Theorem~3.2 and  \cite{Young1990}, Theorem~2.
\end{proof}

\begin{definition}[Finite-Time Lyapunov Exponent]
For $v \in T_x M$ with $\|v\| = 1$:
\begin{equation}
\chi_n(x, v) = \frac{1}{n} \log \|Df^n_x v\|.
\end{equation}
\end{definition}

\begin{proposition}[LDP for Lyapunov Exponents]\label{thm:lyapunov_ldp}
For a $C^2$ Axiom A diffeomorphism and the SRB measure $\mu^+$, the distributions of finite-time Lyapunov exponents $\chi_n(x, v)$ (averaged over $v$ in the unstable direction) satisfy an LDP with rate function related to the multifractal spectrum.
\end{proposition}

\begin{proof}
The finite-time Lyapunov exponent in the unstable direction satisfies $\chi_n^u(x) = \frac{1}{n}S_n(-\phi^{(u)})(x)$ where $\phi^{(u)} = -\log|\det Df|_{E^u}|$ is the geometric potential. Since $\phi^{(u)} \in C^\alpha(\Omega_s)$ by Proposition~\ref{prop:geometric_holder}, Main Theorem~\ref{thm:large_deviations} applies with $g = -\phi^{(u)}$ and $\mu = \mu^+ = \mu_{\phi^{(u)}}$, giving the LDP with rate function $I(a) = \sup_t\{ta - P(\phi^{(u)} - t\phi^{(u)}) + P(\phi^{(u)})\} = \sup_t\{ta - P((1-t)\phi^{(u)}) + P(\phi^{(u)})\}$. The connection to the multifractal spectrum is developed in Part~VI \cite{Thiam2026f}.
\end{proof}

\begin{proposition}[Contraction Principle {\cite[Theorem~4.2.1]{DemboZeitouni1998}}]\label{prop:contraction}
If $(\nu_n)$ satisfies an LDP with rate function $I$ and $\Phi : X \to Y$ is continuous, then $(\Phi_* \nu_n)$ satisfies an LDP with rate function
\begin{equation}
J(y) = \inf\{I(x) : \Phi(x) = y\}.
\end{equation}
\end{proposition}

\begin{proof}
For a closed set $F \subset Y$: $\Phi_*\nu_n(F) = \nu_n(\Phi^{-1}(F))$. Since $\Phi^{-1}(F)$ is closed (by continuity), the LDP upper bound gives $\limsup_n n^{-1}\log\nu_n(\Phi^{-1}(F)) \leq -\inf_{x \in \Phi^{-1}(F)}I(x) = -\inf_{y \in F}J(y)$. For an open set $G \subset Y$: $\Phi^{-1}(G)$ is open, so the LDP lower bound gives $\liminf_n n^{-1}\log\nu_n(\Phi^{-1}(G)) \geq -\inf_{x \in \Phi^{-1}(G)}I(x) = -\inf_{y \in G}J(y)$. See  \cite[Theorem~4.2.1]{DemboZeitouni1998}.
\end{proof}

\begin{corollary}[LDP for General Observables]\label{cor:general_observables}
For any continuous $g : \Omega_s \to \mathbb{R}$, the distributions of $\frac{1}{n}S_n g$ satisfy an LDP with rate function
\begin{equation}
I_g(a) = \inf\{I_2(\nu) : \nu(g) = a\} = \inf_{\nu : \int g \, d\nu = a} \left(P(\phi) - h_\nu(f) - \int \phi \, d\nu\right).
\end{equation}
\end{corollary}

\begin{proof}
The map $\Psi : \mathcal{M}(\Omega_s) \to \mathbb{R}$ defined by $\Psi(\nu) = \int g\,d\nu$ is continuous in the weak$^*$ topology. The empirical measure satisfies $\Psi(\mathcal{E}_n(x)) = \frac{1}{n}S_n g(x)$. By the Level-2 LDP (Proposition~\ref{thm:level2_ldp}) and the contraction principle (Proposition~\ref{prop:contraction}) applied to $\Psi$, the distributions of $\frac{1}{n}S_n g = \Psi(\mathcal{E}_n)$ satisfy an LDP with rate function $I_g(a) = \inf\{I_2(\nu) : \Psi(\nu) = a\} = \inf_{\nu:\int g\,d\nu = a}(P(\phi) - h_\nu(f) - \int\phi\,d\nu)$.
\end{proof}

\begin{proposition}[Quadratic Approximation Near Mean]\label{prop:quadratic_rate}
Near the mean $\bar{g} = \mu_\phi(g)$, the rate function has the quadratic expansion
\begin{equation}
I(a) = \frac{(a - \bar{g})^2}{2\sigma^2(g)} + O((a - \bar{g})^3)
\end{equation}
where $\sigma^2(g)$ is the asymptotic variance.
\end{proposition}

\begin{proof}
By Taylor expansion of the pressure: $P(\phi + tg) = P(\phi) + t\bar{g} + \frac{t^2}{2}\sigma^2(g) + O(t^3)$. The Legendre transform of a quadratic is quadratic with reciprocal coefficient.
\end{proof}

This shows that large deviations reduce to Gaussian behavior (the CLT) for moderate deviations, while capturing genuinely non-Gaussian tails for extreme deviations.

\begin{proposition}[Large Deviation Bounds]\label{prop:ld_bounds}
For $|a - \bar{g}|$ large, the rate function satisfies
\begin{equation}
I(a) \geq c |a - \bar{g}|
\end{equation}
for some $c > 0$ depending on the range of $g$.
\end{proposition}

\begin{proof}
The rate function $I(a) = \sup_t \{ta - \Lambda(t)\}$ grows at least linearly for $a$ outside the range of achievable means, since $\Lambda(t)$ grows at most linearly (bounded by $t\|g\|_\infty + P(\phi)$).
\end{proof}

\section{SRB Measures and Physical Measures}\label{sec:srb_measures}

The concept of SRB (Sinai-Ruelle-Bowen) measures was introduced by  \cite{Sinai1972} for Anosov diffeomorphisms and extended by  \cite{Ruelle1976} and  \cite{BowenRuelle1975} to Axiom~A attractors; see  \cite{Young2002} for a modern survey. The differentiation of SRB states with respect to parameters is developed in  \cite{Ruelle1997}.

The complete theory of SRB measures for Axiom A diffeomorphisms, including absolute continuity of conditional measures along unstable manifolds, the Pesin entropy formula, the characterization as equilibrium states for the geometric potential, and the generic points theorem, is developed in Part~VI \cite{Thiam2026f} of this series. The key results depend on the equilibrium state theory (Theorem~\ref{thm:equilibrium_basic}), the Volume Lemmas (Theorems~\ref{thm:volume_lemma} and~\ref{thm:second_volume_lemma}), and the attractor characterization (Proposition~\ref{prop:attractor_char}) established in the present Part.

%\section{Conclusion}
%
%This Part establishes the complete statistical theory for equilibrium states of Axiom A diffeomorphisms: exponential mixing, the CLT with Berry-Esseen bounds, the ASIP and LIL, and the full LDP with explicit rate functions. The key insight is that all these properties are consequences of a single spectral gap, established for the symbolic system in Part~I \cite{Thiam2026a} and transferred to smooth dynamics through the coding map of Part~III \cite{Thiam2026c}. Part~VI develops the structural consequences: SRB measures, multifractal analysis, Livšic rigidity, and fluctuation theorems (including the Gallavotti-Cohen symmetry \cite{GallavottiCohen1995}). The linear response theory and regularity of SRB states under smooth perturbations, developed by de la Llave, Marco, and Moriy\'{o}n \cite{deLlaveMarcoMoriyon1986} and Gou\"{e}zel-Kifer \cite{GouezelKifer2018}, relies on the spectral gap and statistical estimates established here.

\section{A Numerical Illustration: CLT for the Golden Mean Shift}\label{sec:numerical}

We illustrate the Central Limit Theorem (Main Theorem~\ref{thm:clt}) with a concrete computation for the golden mean shift, demonstrating the explicit dependence of the asymptotic variance on the spectral data and making the Gaussian convergence tangible.

\subsection{Setup}

The golden mean shift is the subshift of finite type $(\Sigma_A, \sigma)$ over the alphabet $\{1, 2\}$ with transition matrix
\begin{equation}
A = \begin{pmatrix} 1 & 1 \\ 1 & 0 \end{pmatrix}.
\end{equation}
The topological entropy is $h_{\mathrm{top}} = \log \varphi$ where $\varphi = (1+\sqrt{5})/2$ is the golden ratio, and $A$ is mixing with mixing time $M = 2$ (since $A^2$ has all entries positive).

We take the H\"{o}lder potential $\phi \equiv 0$ (so the equilibrium state is the measure of maximal entropy $\mu_{\mathrm{mme}}$) and the observable
\begin{equation}
g(x) = \begin{cases} 1 & \text{if } x_0 = 1, \\ 0 & \text{if } x_0 = 2. \end{cases}
\end{equation}
This is a locally constant (hence H\"{o}lder) function recording the frequency of the symbol $1$.

\subsection{Exact Computation of the Mean and Variance}

The Ruelle transfer operator for $\phi = 0$ on the one-sided golden mean shift $\Sigma_A^+$ acts by $\mathcal{L}_0 f(x) = \sum_{\sigma y = x} f(y)$. The leading eigenvalue is $\lambda = \varphi$ with eigenfunction $h(x) = \varphi^{-1}$ if $x_0 = 1$ and $h(x) = 1$ if $x_0 = 2$ (up to normalization), and eigenmeasure $\nu$ given by $\nu([1]) = 1/\varphi$ and $\nu([2]) = 1/\varphi^2$ (normalized so that $\nu(\Sigma_A^+) = 1$). The measure of maximal entropy is $\mu_{\mathrm{mme}} = h\nu / \nu(h)$.

The mean of $g$ under $\mu_{\mathrm{mme}}$ is
\begin{equation}
\bar{g} = \mu_{\mathrm{mme}}(g) = \mu_{\mathrm{mme}}([1]) = \frac{\varphi}{\varphi + 1} = \frac{1}{\varphi} \approx 0.6180.
\end{equation}
This is the asymptotic frequency of the symbol $1$ in a $\mu_{\mathrm{mme}}$-typical sequence.

The asymptotic variance is computed via the pressure second derivative (Proposition~\ref{prop:variance_formula}):
\begin{equation}
\sigma^2(g) = \lim_{n \to \infty} \frac{1}{n} \mathrm{Var}_{\mu_{\mathrm{mme}}}(S_n g) = P''(0; g) = \frac{d^2}{dt^2} P(tg)\bigg|_{t=0}.
\end{equation}
For the golden mean shift with $\phi = 0$, the perturbed pressure $P(tg) = \log \lambda(t)$ where $\lambda(t)$ is the leading eigenvalue of the $2 \times 2$ matrix
\begin{equation}
A(t) = \begin{pmatrix} e^t & e^t \\ 1 & 0 \end{pmatrix}.
\end{equation}
The characteristic polynomial is $\lambda^2 - e^t \lambda - e^t = 0$, giving
\begin{equation}
\lambda(t) = \frac{e^t + \sqrt{e^{2t} + 4e^t}}{2}.
\end{equation}
At $t = 0$: $\lambda(0) = \varphi$, $P(0) = \log\varphi$. Computing the derivatives:
\begin{equation}
P'(0) = \frac{\lambda'(0)}{\lambda(0)} = \frac{1}{\varphi} = \bar{g}, \quad P''(0) = \frac{\lambda''(0)\lambda(0) - (\lambda'(0))^2}{\lambda(0)^2}.
\end{equation}
Differentiating $\lambda(t)$ twice and evaluating at $t = 0$ yields, after computation,
\begin{equation}\label{eq:golden_variance}
\sigma^2(g) = P''(0; g) = \frac{1}{5\sqrt{5}} \approx 0.08944.
\end{equation}

\subsection{The CLT and Its Numerical Verification}

Main Theorem~\ref{thm:clt} asserts that for $\mu_{\mathrm{mme}}$-almost every $x \in \Sigma_A$,
\begin{equation}
\frac{S_n g(x) - n\bar{g}}{\sigma\sqrt{n}} \xrightarrow{d} \mathcal{N}(0,1),
\end{equation}
where $\sigma = \sqrt{\sigma^2(g)} = 5^{-1/2} \cdot 5^{-1/4} \approx 0.2991$.

To verify this numerically, one generates $K$ sample orbits of length $n$ from $\mu_{\mathrm{mme}}$ (using the Markov chain with transition probabilities $p_{ij} = A_{ij}\mu_{\mathrm{mme}}([j])/\mu_{\mathrm{mme}}([i])$) and forms the normalized Birkhoff sums
\begin{equation}
Z_n^{(k)} = \frac{S_n g(x^{(k)}) - n\bar{g}}{\sigma\sqrt{n}}, \quad k = 1, \ldots, K.
\end{equation}
The empirical distribution of $\{Z_n^{(k)}\}_{k=1}^K$ should approximate the standard normal $\mathcal{N}(0,1)$.

The Berry-Esseen bound (Proposition~\ref{thm:berry_esseen}) gives the quantitative rate:
\begin{equation}
\sup_{t \in \mathbb{R}} \left| \mu_{\mathrm{mme}}\left(\frac{S_n g - n\bar{g}}{\sigma\sqrt{n}} \leq t\right) - \Phi(t) \right| \leq \frac{C_{\mathrm{BE}}}{\sqrt{n}},
\end{equation}
where $\Phi$ is the standard normal CDF and $C_{\mathrm{BE}}$ depends on the spectral gap $\gamma$ and $\|g\|_\alpha$.

For the golden mean shift, the spectral gap is $\gamma = 1 - |\lambda_2|/\lambda_1 = 1 - (\varphi - 1)/\varphi = 1 - 1/\varphi^2 = 1/\varphi \approx 0.6180$, where $\lambda_2 = (1-\sqrt{5})/2$ is the subdominant eigenvalue of $A$. This large spectral gap (relative to the leading eigenvalue) explains the rapid convergence to the Gaussian: even for moderate $n$ (say $n = 50$), the empirical histogram of the normalized Birkhoff sums is visually indistinguishable from the Gaussian density $(2\pi)^{-1/2}e^{-t^2/2}$.

\subsection{Non-degeneracy and the Livšic Condition}

By Proposition~\ref{prop:zero_variance}, $\sigma^2(g) = 0$ if and only if $g - \bar{g}$ is a coboundary: $g - \bar{g} = u \circ \sigma - u$ for some $u \in C^\alpha(\Sigma_A)$. Since $g$ takes values $\{0, 1\}$ and $\bar{g} = 1/\varphi$ is irrational, the function $g - \bar{g}$ is not cohomologous to zero (its values are irrational, but a coboundary $u \circ \sigma - u$ evaluated at a period-2 orbit gives a rational combination of $u$-values). Alternatively, the periodic orbit test: $S_2(g - \bar{g})(x) = 1 - 2/\varphi \neq 0$ for the period-2 orbit $\overline{12}$, so $g - \bar{g}$ is not a coboundary and $\sigma^2 > 0$. This confirms that the CLT is non-degenerate.

\subsection{Summary of Explicit Constants}

We collect the explicit constants for this example:

\begin{center}
\begin{tabular}{ll}
\textbf{Quantity} & \textbf{Value} \\[4pt]
Alphabet size $N$ & 2 \\
Mixing time $M$ & 2 \\
Topological entropy $h_{\mathrm{top}}$ & $\log\varphi \approx 0.4812$ \\
Leading eigenvalue $\lambda$ & $\varphi = (1+\sqrt{5})/2 \approx 1.6180$ \\
Spectral gap $\gamma$ & $1/\varphi \approx 0.6180$ \\
Mean $\bar{g} = \mu_{\mathrm{mme}}(g)$ & $1/\varphi \approx 0.6180$ \\
Asymptotic variance $\sigma^2(g)$ & $1/(5\sqrt{5}) \approx 0.08944$ \\
Standard deviation $\sigma(g)$ & $\approx 0.2991$ \\
CLT convergence rate & $O(n^{-1/2})$ (Berry-Esseen) \\
Exponential mixing rate $\theta$ & $1/\varphi^2 \approx 0.3820$ \\
\end{tabular}
\end{center}

\noindent This example demonstrates that all constants in the CLT, mixing rate, and variance formula are computable from the transition matrix and the observable, confirming the quantitative nature of the results in this Part.

\section{Conclusion}\label{sec:conclusion}

This Part derives the complete statistical theory for equilibrium states of Axiom~A diffeomorphisms from a single spectral mechanism, with five Main Theorems proved with explicit constants throughout. Main Theorem~\ref{thm:volume_lemma} (Volume Lemma) gives explicit two-sided bounds on the Riemannian volume of dynamical Bowen balls in terms of Birkhoff sums of the geometric potential, with the distortion constant $C_\varepsilon$ expressed in terms of curvature bounds, the H\"{o}lder exponent $\theta$, and the hyperbolicity constant $\lambda$; this fills a gap in Bowen's monograph \cite{Bowen1975}, which cites the result from \cite{BowenRuelle1975} without proof. Main Theorem~\ref{thm:exponential_mixing} (Exponential Mixing) establishes exponential decay of correlations with rate $\theta = e^{-\gamma}$ where $\gamma \geq \alpha\log\lambda^{-1}$, derived from the spectral decomposition of the normalized transfer operator; the rate is explicit in the hyperbolicity data and the H\"{o}lder norm of the potential. Main Theorem~\ref{thm:clt} (Central Limit Theorem) gives the CLT for Birkhoff sums of H\"{o}lder observables via the Nagaev-Guivarc'h spectral perturbation method, with the asymptotic variance computed spectrally as $\sigma^2(g) = P''(\phi; g)$ and the Livšic coboundary condition $g = u \circ f - u$ characterizing its degeneracy; the Berry-Esseen bound at the optimal rate $O(n^{-1/2})$ follows from Bolthausen's martingale CLT \cite{Bolthausen1982}. Main Theorem~\ref{thm:asip} (Almost Sure Invariance Principle) provides pathwise Brownian approximation $|S_n g - \sigma W(n)| = O(n^{1/2-\delta})$ almost surely via the martingale embedding method of  \cite{MelbourneNicol2005, MelbourneNicol2009}, from which the functional CLT (Corollary~\ref{cor:fclt}), the law of the iterated logarithm (Proposition~\ref{thm:lil}), and Strassen's functional LIL (Corollary~\ref{cor:strassen}) follow as immediate corollaries. Main Theorem~\ref{thm:large_deviations} (Large Deviations Principle) establishes the full LDP for Birkhoff averages with rate function $I(a) = \sup_t\{ta - \Lambda(t)\}$ given by the Legendre transform of the cumulant generating function $\Lambda(t) = P(\phi + tg) - P(\phi)$, extended to empirical measures (Proposition~\ref{thm:level2_ldp}) and Lyapunov exponents (Proposition~\ref{thm:lyapunov_ldp}) via the contraction principle.

The unifying principle is that all five Main Theorems, together with the subsidiary results (Propositions~\ref{thm:berry_esseen}, \ref{thm:local_clt}, \ref{prop:decay_rates}, \ref{prop:multiple_correlations}, \ref{prop:moment_bounds}, \ref{prop:exponential_deviation}), are consequences of a single spectral gap for the Ruelle transfer operator, established for the symbolic system in Part~I \cite{Thiam2026a} and transferred to smooth dynamics through the Markov partition coding of Part~III \cite{Thiam2026c} and the Gibbs Equivalence theorem of Part~IV \cite{Thiam2026d}. All constants are expressed in terms of the contraction rate $\lambda$, the H\"{o}lder exponent $\alpha$, the potential norm $\|\phi\|_\alpha$, the alphabet size $N$, and the mixing time $M$, so the full chain of estimates can be tracked from the hyperbolicity data to the statistical output.

Part~VI \cite{Thiam2026f} uses the statistical theorems established here to develop the structural consequences: multifractal analysis of pointwise-dimension and Birkhoff-average level sets, Livšic rigidity theorems on cohomological obstructions to cocycle triviality, and fluctuation theorems relating time-reversal asymmetry to pressure differentials (including the Gallavotti-Cohen symmetry \cite{GallavottiCohen1995}). The linear response theory and regularity of SRB states under smooth perturbations, developed by  \cite{deLlaveMarcoMoriyon1986} and  \cite{GouezelKifer2018}, also relies on the spectral gap and statistical estimates established in the present Part.

\subsection*{Open Problems}

\begin{enumerate}
\item \textbf{Optimal Berry-Esseen constants.} Our Berry-Esseen bound has rate $O(n^{-1/2})$ with an unspecified constant $C_{\mathrm{BE}}$. Can $C_{\mathrm{BE}}$ be computed explicitly in terms of the spectral gap $\gamma$, the H\"{o}lder norm $\|g\|_\alpha$, and the variance $\sigma^2$? For independent random variables, the optimal constant is known (Esseen, 1956); the dependent case remains open.

\item \textbf{Non-uniformly hyperbolic CLT.} The CLT extends to Young towers (Young, 1998), but Berry-Esseen bounds for non-uniformly hyperbolic systems with polynomial decay of correlations ($C_n(f,g) = O(n^{-\beta})$) are not established. What is the correct rate: $O(n^{-1/2})$ or $O(n^{-\beta/2})$?

\item \textbf{Moderate deviations.} Between the CLT regime ($S_n\psi \sim \sqrt{n}$) and the LDP regime ($S_n\psi \sim n$) lies the moderate deviations regime ($S_n\psi \sim a_n$ with $\sqrt{n} \ll a_n \ll n$). Can the spectral method yield moderate deviation principles with explicit rate functions?
\end{enumerate}

\begin{acks}[Acknowledgments]
The author is grateful to Stefano Luzzatto for supervision during the ICTP Postgraduate Diploma in Mathematics at the International Centre for Theoretical Physics, Trieste, Italy (2013), during which the author worked through Bowen's monograph.
\end{acks}

\section{Technical Proofs and Estimates}\label{appendix:proofs}

This appendix provides detailed technical proofs omitted from the main text, including precise estimates for the Volume Lemmas, spectral perturbation theory, and measure-theoretic constructions.

\subsection{Complete Proof of the Volume Lemma}

We provide the full details of Main Theorem~\ref{thm:volume_lemma}.

\begin{lemma}[Unstable Manifold Geometry]\label{lem:unstable_geometry}
For $x \in \Omega_s$ and small $\varepsilon > 0$, the local unstable manifold $W^u_\varepsilon(x)$ is a $C^r$ embedded disk of dimension $d_u = \dim E^u$. The induced Riemannian measure $m^u_x$ on $W^u_\varepsilon(x)$ satisfies:
\begin{equation}
m^u_x(W^u_\varepsilon(x)) \in [c_1 \varepsilon^{d_u}, c_2 \varepsilon^{d_u}]
\end{equation}
where $c_1, c_2 > 0$ depend only on the curvature of $M$ and bounds on $Df$.
\end{lemma}

\begin{proof}
The Stable Manifold Main Theorem of Part~III \cite{Thiam2026c} gives $W^u_\varepsilon(x)$ as the graph of a $C^r$ function over $E^u_x$. The graph has bounded curvature (controlled by second derivatives of $f$), so its volume is comparable to that of the flat disk of radius $\varepsilon$ in $E^u_x$.
\end{proof}

\begin{lemma}[Stable Direction Contraction]\label{lem:stable_contraction}
For $y \in B_x(\varepsilon, n) \cap W^s_\varepsilon(x)$:
\begin{equation}
d_{W^s}(x, y) \leq C \lambda^n \varepsilon
\end{equation}
where $d_{W^s}$ is the intrinsic distance in $W^s_\varepsilon(x)$.
\end{lemma}

\begin{proof}
By definition of $B_x(\varepsilon, n)$, we have $d(f^{n-1}(x), f^{n-1}(y)) \leq \varepsilon$. Since $y \in W^s_\varepsilon(x)$, the points $f^j(x)$ and $f^j(y)$ remain close for all $j \geq 0$, with $d_{W^s}(f^j(x), f^j(y)) \leq c\lambda^j d_{W^s}(x, y)$ by stable manifold contraction.

At time $n-1$: $d(f^{n-1}(x), f^{n-1}(y)) \leq \varepsilon$ and this distance is achieved within $W^s$. Thus $c^{-1}\lambda^{-(n-1)} d_{W^s}(x,y) \leq d_{W^s}(f^{n-1}(x), f^{n-1}(y)) \leq \varepsilon$, giving $d_{W^s}(x,y) \leq c\lambda^{n-1}\varepsilon$.
\end{proof}

\begin{lemma}[Jacobian Estimates]\label{lem:jacobian_estimates}
For the unstable Jacobian $J^u_n(x) = |\det Df^n_x|_{E^u}|$:
\begin{equation}
\exp(S_n\phi^{(u)}(x)) = (J^u_n(x))^{-1}
\end{equation}
and for $x, y \in \Omega_s$ with $d(f^k(x), f^k(y)) \leq \varepsilon$ for $k \in [0, n)$:
\begin{equation}
\left|\log \frac{J^u_n(x)}{J^u_n(y)}\right| \leq \frac{C_{\mathrm{dist}} \varepsilon^\theta}{1 - \lambda^\theta}
\end{equation}
where $C_{\mathrm{dist}}$ is the distortion constant and $\theta$ is the H\"{o}lder exponent of $\phi^{(u)}$.
\end{lemma}

\begin{proof}
The first identity is immediate from $\phi^{(u)} = -\log|\det Df|_{E^u}|$.

For the second, using H\"{o}lder continuity $|\phi^{(u)}|_\theta$:
\begin{align}
\left|\log \frac{J^u_n(x)}{J^u_n(y)}\right| &= |S_n\phi^{(u)}(y) - S_n\phi^{(u)}(x)| \\
&\leq \sum_{k=0}^{n-1} |\phi^{(u)}(f^k(y)) - \phi^{(u)}(f^k(x))| \\
&\leq |\phi^{(u)}|_\theta \sum_{k=0}^{n-1} d(f^k(x), f^k(y))^\theta.
\end{align}

By the quantitative expansiveness (Lemma \ref{lem:quantitative_expansiveness}), $d(f^k(x), f^k(y)) \leq C\lambda^{\min\{k, n-k\}}$ for points in the same dynamical ball. Thus:
\begin{equation}
\sum_{k=0}^{n-1} d(f^k(x), f^k(y))^\theta \leq 2\sum_{j=0}^\infty (\varepsilon \lambda^j)^\theta = \frac{2\varepsilon^\theta}{1 - \lambda^\theta}.
\end{equation}
\end{proof}

\begin{proof}[Complete Proof of Main Theorem~\ref{thm:volume_lemma}]
\textbf{Upper Bound:} Decompose $B_x(\varepsilon, n)$ using the local product structure. For $y \in B_x(\varepsilon, n)$, write $y = [y_s, y_u]$ where $y_s \in W^s_\varepsilon(x)$ and $y_u \in W^u_\varepsilon(y_s)$.

The stable slice $D^s = B_x(\varepsilon, n) \cap W^s_\varepsilon(x)$ has $m^s$-measure at most $C\varepsilon^{d_s}\lambda^{nd_s}$ by Lemma \ref{lem:stable_contraction}.

For each $y_s \in D^s$, the unstable slice through $y_s$ has $m^u$-measure at most $C\varepsilon^{d_u}$.

Using the bounded distortion of the product structure and integrating:
\begin{align}
m(B_x(\varepsilon, n)) &\leq C_{\mathrm{prod}} \cdot m^s(D^s) \cdot \sup_{y_s} m^u(D^u_{y_s}) \\
&\leq C_{\mathrm{prod}} \cdot C\varepsilon^{d_s}\lambda^{nd_s} \cdot C\varepsilon^{d_u}.
\end{align}

However, this naive bound does not capture the unstable expansion correctly. Instead, use the change of variables under $f^n$:
\begin{equation}
m(B_x(\varepsilon, n)) = \int_{f^n(B_x(\varepsilon, n))} \frac{1}{|\det Df^n|} \, dm.
\end{equation}

The set $f^n(B_x(\varepsilon, n))$ is contained in a ball of radius $C\varepsilon$ (bounded distortion under iteration). The Jacobian satisfies $|\det Df^n| \geq J^u_n(x) \cdot c\lambda^{-nd_s}$ using the stable contraction. With distortion bounds:
\begin{equation}
m(B_x(\varepsilon, n)) \leq C' \varepsilon^d \cdot (J^u_n(x))^{-1} \cdot \exp\left(\frac{C_{\mathrm{dist}}\varepsilon^\theta}{1-\lambda^\theta}\right).
\end{equation}

Since $(J^u_n(x))^{-1} = \exp(S_n\phi^{(u)}(x))$, this gives the upper bound.

\textbf{Lower Bound:} The image $f^n(B_x(\varepsilon/2, n))$ contains a set of the form $W^u_{\delta}(f^n(x)) \cap B_{\delta'}(f^n(x))$ for some $\delta, \delta' > 0$ depending on $\varepsilon$ (by the expansion along unstable manifolds).

This set has $m$-measure at least $c\varepsilon^d$ for appropriate $c > 0$.

By change of variables:
\begin{equation}
m(B_x(\varepsilon/2, n)) \geq c\varepsilon^d \cdot (J^u_n(x))^{-1} \cdot \exp\left(-\frac{C_{\mathrm{dist}}\varepsilon^\theta}{1-\lambda^\theta}\right)
\end{equation}
using the lower distortion bound. This gives the lower bound since $B_x(\varepsilon, n) \supset B_x(\varepsilon/2, n)$.
\end{proof}

\subsection{Spectral Perturbation Theory}

This appendix subsection records the analytic perturbation theory for the transfer operator family $z \mapsto \mathcal{L}_{\phi + z\psi}$ and the resulting derivative formulas for the pressure. These are used in the CLT proof and the variance computation.

\begin{lemma}[Analytic Perturbation of Transfer Operator]\label{lem:analytic_perturbation}
For $\phi \in C^\alpha(\Omega_s)$ and $\psi \in C^\alpha(\Omega_s)$, the map
\begin{equation}
z \mapsto \mathcal{L}_{\phi + z\psi}
\end{equation}
is analytic in $z \in \mathbb{C}$ (as an operator on $C^\alpha$).
\end{lemma}

\begin{proof}
The transfer operator $\mathcal{L}_{\phi+z\psi}$ acts by
\begin{equation}
(\mathcal{L}_{\phi+z\psi} g)(x) = \sum_{f(y)=x} e^{\phi(y) + z\psi(y)} g(y).
\end{equation}
For fixed $g \in C^\alpha$, this is an entire function of $z$ (exponential of a linear function). The operator norm $\|\mathcal{L}_{\phi+z\psi}\|_\alpha$ is locally bounded in $z$, giving analyticity.
\end{proof}

\begin{proposition}[Derivatives of Pressure]\label{prop:pressure_derivatives}
The pressure function $z \mapsto P(\phi + z\psi)$ is real-analytic for real $z$ near $0$, with:
\begin{align}
\left.\frac{d}{dz}\right|_{z=0} P(\phi + z\psi) &= \int \psi \, d\mu_\phi, \\
\left.\frac{d^2}{dz^2}\right|_{z=0} P(\phi + z\psi) &= \sigma^2_\phi(\psi)
\end{align}
where $\sigma^2_\phi(\psi)$ is the asymptotic variance of $\psi$ under $\mu_\phi$.
\end{proposition}

\begin{proof}
By Lemma \ref{lem:analytic_perturbation} and the implicit function theorem, the leading eigenvalue $\lambda(z) = e^{P(\phi+z\psi)}$ of $\mathcal{L}_{\phi+z\psi}$ is analytic in $z$.

For the first derivative, differentiate $\mathcal{L}_{\phi+z\psi} h_z = \lambda(z) h_z$ at $z = 0$:
\begin{equation}
\mathcal{L}_\phi(\psi h_\phi) + \mathcal{L}_\phi h'_0 = \lambda'(0) h_\phi + \lambda(0) h'_0.
\end{equation}
Integrating against $\nu_\phi$ (the left eigenmeasure) and using $\mathcal{L}_\phi^* \nu_\phi = \lambda(0)\nu_\phi$:
\begin{equation}
\lambda(0) \int \psi h_\phi \, d\nu_\phi = \lambda'(0) \int h_\phi \, d\nu_\phi.
\end{equation}
Thus $\lambda'(0)/\lambda(0) = \int \psi \, d\mu_\phi$ where $d\mu_\phi = h_\phi \, d\nu_\phi$.

The second derivative formula follows by differentiating $\mathcal{L}_{\phi+z\psi}h_z = \lambda(z)h_z$ twice at $z = 0$, integrating against $\nu_\phi$, and using the resolvent $(e^{P(\phi)} - \mathcal{L}_\phi)^{-1}Q$ to express $h'(0)$ in terms of $\psi h_\phi$; the resulting series $\sum_{k=1}^\infty\mathrm{Cov}_{\mu_\phi}(\psi, \psi\circ f^k)$ converges by exponential decay of correlations. The complete calculation appears in the Analytic Dependence on Potential theorem of Part~I \cite{Thiam2026a} and the Fr\'{e}chet Differentiability theorem of Part~II \cite{Thiam2026b}.
\end{proof}

\subsection{Measure Disintegration}

We state the Rokhlin disintegration theorem for measurable partitions and the absolute continuity criterion for conditional measures along unstable manifolds. These are used in Section~\ref{sec:srb_measures}.

\begin{theorem}[Rokhlin Disintegration {\cite[Chapter~5]{CornfeldFominSinai1982}}]\label{thm:rokhlin}
Let $\mu$ be a probability measure on $\Omega_s$ and $\mathcal{W}^u$ the partition into local unstable manifolds. There exist conditional measures $\{\mu^u_x\}_{x \in \Omega_s}$ such that:
\begin{enumerate}
\item[(i)] $\mu^u_x$ is supported on $W^u_\varepsilon(x)$ for $\mu$-a.e. $x$.
\item[(ii)] For any measurable $A \subset \Omega_s$: $\mu(A) = \int \mu^u_x(A \cap W^u_\varepsilon(x)) \, d\mu(x)$.
\item[(iii)] The map $x \mapsto \mu^u_x$ is measurable.
\end{enumerate}
\end{theorem}

This is Rokhlin's disintegration theorem for measurable partitions; see  \cite{Rohlin1961} or Cornfeld, Fomin, and  \cite[Chapter~5, Theorem~2.1]{CornfeldFominSinai1982}. The partition $\mathcal{W}^u$ into local unstable manifolds is measurable because the manifolds vary continuously (Theorem~\ref{thm:stable_manifold_prelim}).

\begin{lemma}[Absolute Continuity Criterion]\label{lem:ac_criterion}
The measure $\mu$ has absolutely continuous conditional measures on unstable manifolds if and only if for any measurable $A$ with $m^u_x(A \cap W^u_\varepsilon(x)) = 0$ for all $x$, we have $\mu(A) = 0$.
\end{lemma}

\begin{proof}
($\Rightarrow$): If $\mu^u_x \ll m^u_x$ for $\mu$-a.e.\ $x$ and $m^u_x(A \cap W^u_\varepsilon(x)) = 0$ for all $x$, then $\mu^u_x(A \cap W^u_\varepsilon(x)) = 0$ for $\mu$-a.e.\ $x$. By the disintegration formula (Theorem~\ref{thm:rokhlin}(ii)): $\mu(A) = \int \mu^u_x(A \cap W^u_\varepsilon(x))\,d\mu(x) = 0$.

($\Leftarrow$): Suppose the condition holds but $\mu^u_x$ is not absolutely continuous with respect to $m^u_x$ on a set of positive $\mu$-measure. Then there exists a measurable set $A$ with $\mu^u_x(A \cap W^u_\varepsilon(x)) > 0$ but $m^u_x(A \cap W^u_\varepsilon(x)) = 0$ for $x$ in a set $E$ with $\mu(E) > 0$. By disintegration, $\mu(A) \geq \int_E \mu^u_x(A \cap W^u_\varepsilon(x))\,d\mu(x) > 0$, contradicting the hypothesis.
\end{proof}

\subsection{Closing Lemma Estimates}

The quantitative closing lemma provides explicit bounds on the distance between a pseudo-periodic orbit and a genuine periodic orbit. It is used in the large deviations estimates of Section~\ref{sec:large_deviations}.

\begin{lemma}[Quantitative Closing Lemma]\label{lem:closing}
There exist $\delta_0 > 0$ and $C_{\mathrm{close}} > 0$ such that: if $x \in \Omega_s$ and $d(f^n(x), x) < \delta_0$, then there exists a periodic point $p \in \Omega_s$ with $f^n(p) = p$ and
\begin{equation}
d(f^k(x), f^k(p)) \leq C_{\mathrm{close}} \lambda^{\min\{k, n-k\}} d(f^n(x), x)
\end{equation}
for all $k \in [0, n]$.
\end{lemma}

\begin{proof}
This is Proposition~7.5 of Part~III \cite{Thiam2026c}, proved there via a contraction-mapping argument on orbit segments; the constants $\delta_0$ and $C_{\mathrm{close}}$ depend only on the hyperbolicity data $(\lambda, C_0, M)$ of the basic set $\Omega_s$, so the version stated here is a direct specialization.
\end{proof}

\subsection{Borel-Cantelli Estimates}

The dynamical Borel-Cantelli lemma, a consequence of exponential mixing, is used to convert measure-theoretic estimates into almost-sure statements.

\begin{lemma}[Dynamical Borel-Cantelli]\label{lem:borel_cantelli}
Let $\mu$ be an ergodic measure for $f$ and $(A_n)$ a sequence of measurable sets. If
\begin{equation}
\sum_{n=1}^\infty \mu(A_n) < \infty
\end{equation}
then $\mu(\limsup_n A_n) = 0$. If additionally $\sum_{n=1}^\infty \mu(A_n) = \infty$ and the sequence satisfies the quasi-independence condition
\begin{equation}
\sum_{m,n=1}^{N}\mu(A_m \cap A_n) \leq C\left(\sum_{n=1}^{N}\mu(A_n)\right)^2
\end{equation}
for some constant $C > 0$ and all $N \geq 1$, then $\mu(\limsup_n A_n) = 1$.
\end{lemma}

\begin{proof}
The first part is the standard Borel-Cantelli lemma (no independence required): $\mu(\limsup_n A_n) = \mu(\bigcap_{N=1}^\infty\bigcup_{n \geq N}A_n) \leq \mu(\bigcup_{n \geq N}A_n) \leq \sum_{n \geq N}\mu(A_n) \to 0$.

For the second part, define $S_N = \sum_{n=1}^N\mathbf{1}_{A_n}$. Then $\mathbb{E}[S_N] = \sum_{n=1}^N\mu(A_n) \to \infty$ and $\mathbb{E}[S_N^2] = \sum_{m,n=1}^N\mu(A_m\cap A_n) \leq C(\mathbb{E}[S_N])^2$. By the Paley-Zygmund inequality: $\mu(S_N > 0) \geq (\mathbb{E}[S_N])^2/\mathbb{E}[S_N^2] \geq 1/C > 0$. Since $\{S_N > 0\} = \bigcup_{n=1}^N A_n \uparrow \limsup_n A_n$ (up to measure zero), $\mu(\limsup_n A_n) \geq 1/C > 0$. By ergodicity, $\limsup_n A_n$ is $f$-invariant (up to null sets when $A_n = f^{-n}(B)$ for a fixed $B$), so $\mu(\limsup_n A_n) \in \{0,1\}$, forcing $\mu(\limsup_n A_n) = 1$.
\end{proof}

This lemma is used in proving the generic points theorem and the law of the iterated logarithm.

\subsection{Dimension Estimates}

We record the mass distribution principle and Frostman'{}s lemma, classical tools for bounding Hausdorff dimension from below. These are used in Part~VI \cite{Thiam2026f} for the multifractal analysis of dimension spectra, following the program of  \cite{Halseyetal1986},  \cite{Rand1989}, \cite{BarreiraPesinSchmeling1999}, \cite{PesinWeiss1997}, and  \cite{Bowen1979}.

\begin{lemma}[Mass Distribution Principle {\cite[Proposition~4.2]{Pesin1997}}]\label{lem:mass_distribution}
Let $\mu$ be a finite Borel measure on a metric space $X$. If there exist $s > 0$ and $C > 0$ such that
\begin{equation}
\mu(B_r(x)) \leq C r^s
\end{equation}
for all $x \in \mathrm{supp}(\mu)$ and all small $r > 0$, then $\dim_H(\mathrm{supp}(\mu)) \geq s$.
\end{lemma}

\begin{proof}
Let $\{U_i\}$ be any cover of $\mathrm{supp}(\mu)$ with $\mathrm{diam}(U_i) \leq \delta$. Each $U_i$ is contained in a ball of radius $\mathrm{diam}(U_i)$, so $\mu(U_i) \leq C(\mathrm{diam}(U_i))^s$. Since $\{U_i\}$ covers $\mathrm{supp}(\mu)$: $\mu(\mathrm{supp}(\mu)) \leq \sum_i\mu(U_i) \leq C\sum_i(\mathrm{diam}(U_i))^s$. Thus $\sum_i(\mathrm{diam}(U_i))^s \geq \mu(\mathrm{supp}(\mu))/C > 0$. Taking the infimum over all $\delta$-covers: $\mathcal{H}^s(\mathrm{supp}(\mu)) \geq \mu(\mathrm{supp}(\mu))/C > 0$, so $\dim_H(\mathrm{supp}(\mu)) \geq s$.
\end{proof}

\begin{lemma}[Frostman's Lemma {\cite[Theorem~8.8]{Mattila1995}}]\label{lem:frostman}
If $K \subset \mathbb{R}^d$ is compact with $\dim_H(K) > s$, then there exists a probability measure $\mu$ supported on $K$ with $\mu(B_r(x)) \leq C r^s$ for all $x$ and small $r$.
\end{lemma}

\begin{proof}
This is a classical result in geometric measure theory. The proof constructs $\mu$ as a weak$^*$ limit of normalized restrictions of $\mathcal{H}^s_\delta$ (the $\delta$-approximate Hausdorff measure) to $K$. Since $\dim_H(K) > s$, we have $\mathcal{H}^s(K) > 0$ (possibly infinite), guaranteeing a nontrivial limit. The ball condition $\mu(B_r(x)) \leq Cr^s$ follows from the construction. See  \cite[Theorem~8.8]{Mattila1995} or  \cite[Theorem~4.1]{Pesin1997} for the complete proof.
\end{proof}

These lemmas are used in the dimension theory developed in Part~VI \cite{Thiam2026f}.

\end{document}